\documentclass[12pt]{article}
\usepackage{amsmath, amssymb}
\usepackage{amsthm}

\usepackage{hyperref}
\usepackage{verbatim}
\usepackage[top=1.0in, bottom=1.0in, left=1.0in, right=1.0in]{geometry}

\usepackage{tkz-graph}
\usetikzlibrary{arrows}
\usetikzlibrary{shapes}
\usepackage[position=bottom]{subfig}

\usepackage{longtable}
\usepackage{array}

\usepackage{sectsty}
\allsectionsfont{\sffamily}

\makeatletter
\newtheorem*{rep@theorem}{\rep@title}
\newcommand{\newreptheorem}[2]{
\newenvironment{rep#1}[1]{
 \def\rep@title{#2 \ref{##1}}
 \begin{rep@theorem}}
 {\end{rep@theorem}}}
\makeatother

\theoremstyle{plain}
\newtheorem{thm}{Theorem}[section]
\newreptheorem{thm}{Theorem}

\newreptheorem{prop}{Proposition}
\newtheorem{lem}[thm]{Lemma}
\newreptheorem{lem}{Lemma}
\newtheorem{conjecture}[thm]{Conjecture}
\newreptheorem{conjecture}{Conjecture}

\newreptheorem{cor}{Corollary}

\newtheorem*{SmallPotLemma}{Small Pot Lemma}

\theoremstyle{definition}

\theoremstyle{remark}

\newtheorem*{question}{Question}

\newcommand{\IN}{\mathbb{N}}

\newcommand{\claim}[2]{{\bf Claim #1.}~{\it #2}~~}
\newcommand{\claimnonum}[1]{{\bf Claim.}~{\it #1}~~}
\newcommand{\subclaim}[2]{{\bf Subclaim #1.}~{\it #2}~~}

\newcommand{\set}[1]{\left\{ #1 \right\}}
\newcommand{\setb}[3]{\left\{ #1 \in #2 \mid #3 \right\}}
\newcommand{\setbs}[2]{\left\{ #1 \mid #2 \right\}}
\newcommand{\card}[1]{\left|#1\right|}

\newcommand{\ceil}[1]{\left\lceil#1\right\rceil}

\newcommand{\func}[3]{#1\colon #2 \rightarrow #3}

\newcommand{\irange}[1]{\left[#1\right]}
\newcommand{\join}[2]{#1 \mbox{\hspace{2 pt}$\ast$\hspace{2 pt}} #2}
\newcommand{\djunion}[2]{#1 \mbox{\hspace{2 pt}$+$\hspace{2 pt}} #2}
\newcommand{\parens}[1]{\left( #1 \right)}

\newcommand{\DefinedAs}{\mathrel{\mathop:}=}

\def\adj{\leftrightarrow}
\def\nonadj{\not\!\leftrightarrow}

\title{Coloring claw-free graphs with $\Delta-1$ colors} \author{Daniel W.
Cranston\thanks{Department of Mathematics and Applied Mathematics, Virginia Commonwealth University, Richmond, VA, 23284. email: \texttt{dcranston@vcu.edu}} \and Landon Rabern\thanks{School of Mathematical and Statistical Sciences, Arizona
State University, Tempe, AZ 85287.  email: \texttt{landon.rabern@gmail.com}. 
}
}

\begin{document}
\maketitle
\begin{abstract}
We prove that every claw-free graph $G$ that does not contain a clique on
$\Delta(G) \geq 9$ vertices can be $\Delta(G) - 1$ colored.
\end{abstract}

\section{Introduction}

The first non-trivial result about coloring graphs with around $\Delta$ colors
is Brooks' theorem from 1941. \begin{thm}[Brooks \cite{brooks1941colouring}]
Every graph with $\Delta \geq 3$ satisfies $\chi \leq \max\{\omega, \Delta\}$.
\end{thm}

In 1977, Borodin and Kostochka conjectured that a similar result holds for $\Delta - 1$ colorings.  

\begin{conjecture}[Borodin and Kostochka \cite{borodin1977upper}]
Every graph with $\Delta \geq 9$ satisfies $\chi \leq \max\{\omega, \Delta - 1\}$.
\end{conjecture}

Graphs exist (see Figure \ref{fig:SmallCE}) showing that the $\Delta \geq 9$
condition is necessary.  Using probabilistic methods, Reed
\cite{reed1999strengthening} proved the conjecture for $\Delta \geq
10^{14}$.

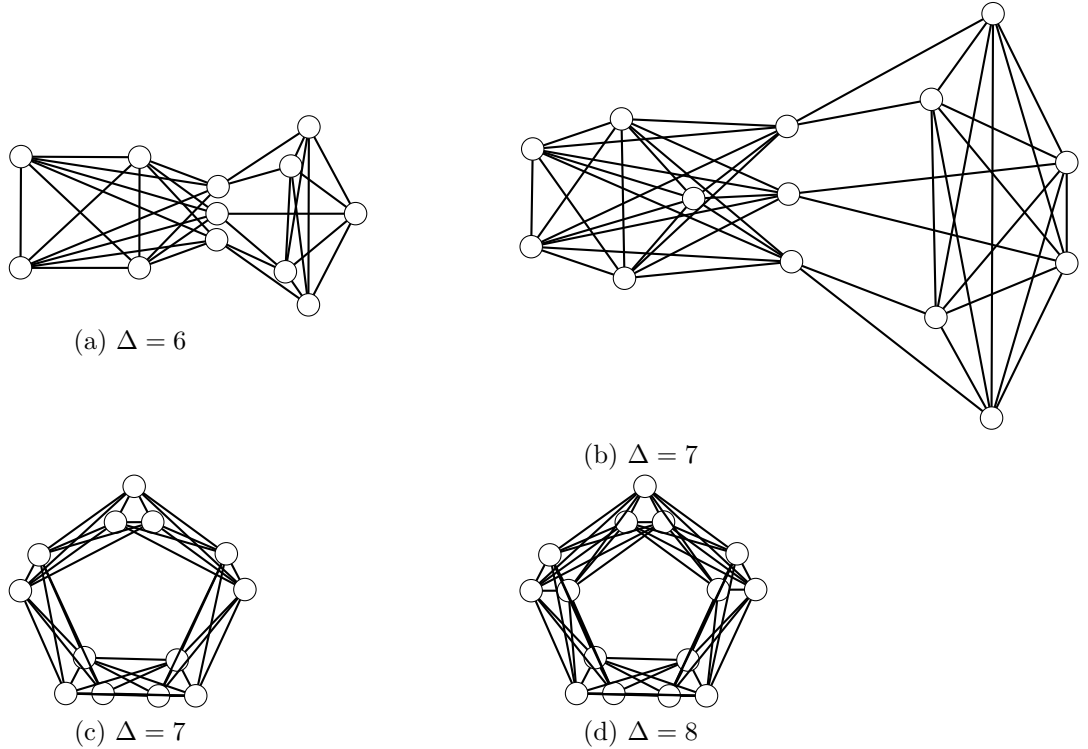
\begin{figure}[htb]
\centering
\subfloat[$\Delta=6$~~~~~~~~~~~~~~~]{
{\parbox{5cm}{
\begin{tikzpicture}[scale = 10]
\tikzstyle{VertexStyle}=[shape = circle, minimum size = 1pt, inner sep = 3pt,
draw]
\Vertex[x = 0.25085711479187, y = 0.92838092893362, L = \tiny {}]{v0}
\Vertex[x = 0.0932380929589272, y = 0.929142817854881, L = \tiny {}]{v1}
\Vertex[x = 0.250571489334106, y = 0.781142801046371, L = \tiny {}]{v2}
\Vertex[x = 0.092571459710598, y = 0.781142845749855, L = \tiny {}]{v3}
\Vertex[x = 0.355238050222397, y = 0.889142841100693, L = \tiny {}]{v4}
\Vertex[x = 0.353904783725739, y = 0.853142827749252, L = \tiny {}]{v5}
\Vertex[x = 0.353238135576248, y = 0.818476170301437, L = \tiny {}]{v6}
\Vertex[x = 0.476000010967255, y = 0.968571435660124, L = \tiny {}]{v7}
\Vertex[x = 0.537999987602234, y = 0.853238105773926, L = \tiny {}]{v8}
\Vertex[x = 0.444666564464569, y = 0.77590474486351, L = \tiny {}]{v9}
\Vertex[x = 0.475333333015442, y = 0.731904745101929, L = \tiny {}]{v10}
\Vertex[x = 0.451999962329865, y = 0.916571423411369, L = \tiny {}]{v11}
\Edge[](v0)(v1)
\Edge[](v2)(v1)
\Edge[](v3)(v1)
\Edge[](v0)(v3)
\Edge[](v2)(v3)
\Edge[](v2)(v0)
\Edge[](v4)(v2)
\Edge[](v5)(v2)
\Edge[](v6)(v2)
\Edge[](v4)(v0)
\Edge[](v5)(v0)
\Edge[](v6)(v0)
\Edge[](v4)(v1)
\Edge[](v5)(v1)
\Edge[](v6)(v1)
\Edge[](v4)(v3)
\Edge[](v5)(v3)
\Edge[](v6)(v3)
\Edge[](v8)(v7)
\Edge[](v9)(v7)
\Edge[](v10)(v7)
\Edge[](v11)(v7)
\Edge[](v9)(v8)
\Edge[](v10)(v8)
\Edge[](v11)(v8)
\Edge[](v10)(v9)
\Edge[](v11)(v9)
\Edge[](v11)(v10)
\Edge[](v4)(v7)
\Edge[](v4)(v11)
\Edge[](v6)(v9)
\Edge[](v6)(v10)
\Edge[](v5)(v8)
\Edge[](v5)(v9)
\end{tikzpicture}}}}\qquad\qquad
\subfloat[$\Delta=7$~~~~~~~~~~~~~~~]{
{\parbox{5cm}{

\begin{tikzpicture}[scale = 10]
\tikzstyle{VertexStyle}=[shape = circle, minimum size = 1pt, inner sep = 3pt,
draw]
\Vertex[x = 0.751999914646149, y = 0.724000096321106, L = \tiny {}]{v0}
\Vertex[x = 0.751999974250793, y = 0.590000092983246, L = \tiny {}]{v1}
\Vertex[x = 0.652000069618225, y = 0.38400000333786, L = \tiny {}]{v2}
\Vertex[x = 0.578000009059906, y = 0.51800012588501, L = \tiny {}]{v3}
\Vertex[x = 0.572000086307526, y = 0.808000028133392, L = \tiny {}]{v4}
\Vertex[x = 0.0419999808073044, y = 0.742000013589859, L = \tiny {}]{v5}
\Vertex[x = 0.0399999916553497, y = 0.612000048160553, L = \tiny {}]{v6}
\Vertex[x = 0.163999989628792, y = 0.569999992847443, L = \tiny {}]{v7}
\Vertex[x = 0.25600004196167, y = 0.676000028848648, L = \tiny {}]{v8}
\Vertex[x = 0.159999996423721, y = 0.782000005245209, L = \tiny {}]{v9}
\Vertex[x = 0.653999924659729, y = 0.921999998390675, L = \tiny {}]{v10}
\Vertex[x = 0.379999995231628, y = 0.771999999880791, L = \tiny {}]{v11}
\Vertex[x = 0.381999999284744, y = 0.682000011205673, L = \tiny {}]{v12}
\Vertex[x = 0.386000007390976, y = 0.592000007629395, L = \tiny {}]{v13}
\Edge[](v0)(v4)
\Edge[](v1)(v4)
\Edge[](v2)(v4)
\Edge[](v3)(v4)
\Edge[](v0)(v3)
\Edge[](v1)(v3)
\Edge[](v2)(v3)
\Edge[](v0)(v2)
\Edge[](v1)(v2)
\Edge[](v0)(v1)
\Edge[](v5)(v6)
\Edge[](v5)(v7)
\Edge[](v5)(v8)
\Edge[](v5)(v9)
\Edge[](v6)(v7)
\Edge[](v6)(v8)
\Edge[](v6)(v9)
\Edge[](v7)(v8)
\Edge[](v7)(v9)
\Edge[](v8)(v9)
\Edge[](v0)(v10)
\Edge[](v1)(v10)
\Edge[](v2)(v10)
\Edge[](v3)(v10)
\Edge[](v4)(v10)
\Edge[](v5)(v11)
\Edge[](v6)(v11)
\Edge[](v7)(v11)
\Edge[](v8)(v11)
\Edge[](v9)(v11)
\Edge[](v5)(v12)
\Edge[](v6)(v12)
\Edge[](v7)(v12)
\Edge[](v8)(v12)
\Edge[](v9)(v12)
\Edge[](v5)(v13)
\Edge[](v6)(v13)
\Edge[](v7)(v13)
\Edge[](v8)(v13)
\Edge[](v9)(v13)
\Edge[](v11)(v10)
\Edge[](v11)(v4)
\Edge[](v12)(v0)
\Edge[](v12)(v1)
\Edge[](v13)(v3)
\Edge[](v13)(v2)
\end{tikzpicture}}}}\qquad\qquad
\subfloat[$\Delta=7$~~~~~~~~~~~~~~~]{
{\parbox{5cm}{

\begin{tikzpicture}[scale = 10]
\tikzstyle{VertexStyle}=[shape = circle, minimum size = 1pt, inner sep = 3pt,
draw]
\Vertex[x = 0.257401078939438, y = 0.729450404644012, L = \tiny {}]{v0}
\Vertex[x = 0.232565611600876, y = 0.681758105754852, L = \tiny {}]{v1}
\Vertex[x = 0.383801102638245, y = 0.820650428533554, L = \tiny {}]{v2}
\Vertex[x = 0.358965694904327, y = 0.772958129644394, L = \tiny {}]{v3}
\Vertex[x = 0.40850430727005, y = 0.773111909627914, L = \tiny {}]{v4}
\Vertex[x = 0.506290018558502, y = 0.730872631072998, L = \tiny {}]{v5}
\Vertex[x = 0.530993163585663, y = 0.683334112167358, L = \tiny {}]{v6}
\Vertex[x = 0.440956592559814, y = 0.589494824409485, L = \tiny {}]{v7}
\Vertex[x = 0.416121125221252, y = 0.541802525520325, L = \tiny {}]{v8}
\Vertex[x = 0.46565979719162, y = 0.541956305503845, L = \tiny {}]{v9}
\Vertex[x = 0.317756593227386, y = 0.592694818973541, L = \tiny {}]{v10}
\Vertex[x = 0.292921125888824, y = 0.545002520084381, L = \tiny {}]{v11}
\Vertex[x = 0.342459797859192, y = 0.545156300067902, L = \tiny {}]{v12}
\Edge[](v1)(v0)
\Edge[](v3)(v2)
\Edge[](v4)(v2)
\Edge[](v4)(v3)
\Edge[](v6)(v5)
\Edge[](v8)(v7)
\Edge[](v9)(v7)
\Edge[](v9)(v8)
\Edge[](v11)(v10)
\Edge[](v12)(v10)
\Edge[](v12)(v11)
\Edge[](v2)(v0)
\Edge[](v3)(v0)
\Edge[](v4)(v0)
\Edge[](v2)(v1)
\Edge[](v3)(v1)
\Edge[](v4)(v1)
\Edge[](v2)(v5)
\Edge[](v3)(v5)
\Edge[](v4)(v5)
\Edge[](v2)(v6)
\Edge[](v3)(v6)
\Edge[](v4)(v6)
\Edge[](v5)(v7)
\Edge[](v6)(v7)
\Edge[](v5)(v9)
\Edge[](v6)(v9)
\Edge[](v5)(v8)
\Edge[](v6)(v8)
\Edge[](v10)(v8)
\Edge[](v11)(v8)
\Edge[](v12)(v8)
\Edge[](v10)(v7)
\Edge[](v11)(v7)
\Edge[](v12)(v7)
\Edge[](v10)(v9)
\Edge[](v11)(v9)
\Edge[](v12)(v9)
\Edge[](v10)(v1)
\Edge[](v11)(v1)
\Edge[](v12)(v1)
\Edge[](v10)(v0)
\Edge[](v11)(v0)
\Edge[](v12)(v0)
\end{tikzpicture}}}}\qquad\qquad
\subfloat[$\Delta=8$~~~~~~~~~~~~~~~]{
{\parbox{5cm}{

\begin{tikzpicture}[scale = 10]
\tikzstyle{VertexStyle}=[shape = circle, minimum size = 1pt, inner sep = 3pt,
draw]
\Vertex[x = 0.257401078939438, y = 0.729450404644012, L = \tiny {}]{v0}
\Vertex[x = 0.232565611600876, y = 0.681758105754852, L = \tiny {}]{v1}
\Vertex[x = 0.282104313373566, y = 0.681911885738373, L = \tiny {}]{v2}
\Vertex[x = 0.383801102638245, y = 0.820650428533554, L = \tiny {}]{v3}
\Vertex[x = 0.358965694904327, y = 0.772958129644394, L = \tiny {}]{v4}
\Vertex[x = 0.40850430727005, y = 0.773111909627914, L = \tiny {}]{v5}
\Vertex[x = 0.506290018558502, y = 0.730872631072998, L = \tiny {}]{v6}
\Vertex[x = 0.48145455121994, y = 0.683180332183838, L = \tiny {}]{v7}
\Vertex[x = 0.530993163585663, y = 0.683334112167358, L = \tiny {}]{v8}
\Vertex[x = 0.440956592559814, y = 0.589494824409485, L = \tiny {}]{v9}
\Vertex[x = 0.416121125221252, y = 0.541802525520325, L = \tiny {}]{v10}
\Vertex[x = 0.46565979719162, y = 0.541956305503845, L = \tiny {}]{v11}
\Vertex[x = 0.317756593227386, y = 0.592694818973541, L = \tiny {}]{v12}
\Vertex[x = 0.292921125888824, y = 0.545002520084381, L = \tiny {}]{v13}
\Vertex[x = 0.342459797859192, y = 0.545156300067902, L = \tiny {}]{v14}
\Edge[](v2)(v1)
\Edge[](v2)(v0)
\Edge[](v1)(v0)
\Edge[](v4)(v3)
\Edge[](v5)(v3)
\Edge[](v5)(v4)
\Edge[](v7)(v6)
\Edge[](v8)(v6)
\Edge[](v8)(v7)
\Edge[](v10)(v9)
\Edge[](v11)(v9)
\Edge[](v11)(v10)
\Edge[](v13)(v12)
\Edge[](v14)(v12)
\Edge[](v14)(v13)
\Edge[](v3)(v0)
\Edge[](v4)(v0)
\Edge[](v5)(v0)
\Edge[](v3)(v2)
\Edge[](v4)(v2)
\Edge[](v5)(v2)
\Edge[](v3)(v1)
\Edge[](v4)(v1)
\Edge[](v5)(v1)
\Edge[](v3)(v6)
\Edge[](v4)(v6)
\Edge[](v5)(v6)
\Edge[](v3)(v7)
\Edge[](v4)(v7)
\Edge[](v5)(v7)
\Edge[](v3)(v8)
\Edge[](v4)(v8)
\Edge[](v5)(v8)
\Edge[](v6)(v9)
\Edge[](v7)(v9)
\Edge[](v8)(v9)
\Edge[](v6)(v11)
\Edge[](v7)(v11)
\Edge[](v8)(v11)
\Edge[](v6)(v10)
\Edge[](v7)(v10)
\Edge[](v8)(v10)
\Edge[](v12)(v10)
\Edge[](v13)(v10)
\Edge[](v14)(v10)
\Edge[](v12)(v9)
\Edge[](v13)(v9)
\Edge[](v14)(v9)
\Edge[](v12)(v11)
\Edge[](v13)(v11)
\Edge[](v14)(v11)
\Edge[](v12)(v2)
\Edge[](v13)(v2)
\Edge[](v14)(v2)
\Edge[](v12)(v1)
\Edge[](v13)(v1)
\Edge[](v14)(v1)
\Edge[](v12)(v0)
\Edge[](v13)(v0)
\Edge[](v14)(v0)
\end{tikzpicture}}}}
\caption{Counterexamples to the Borodin-Kostochka Conjecture for small
$\Delta$.}
\label{fig:SmallCE}
\end{figure}

In \cite{dhurandhar1982improvement}, Dhurandhar proved the Borodin-Kostochka
Conjecture for a superset of line graphs of \emph{simple} graphs defined by excluding the claw, $K_5 - e$, and another graph $D$ as induced subgraphs.  
Kierstead and Schmerl \cite{kierstead1986chromatic} improved this by removing
the need to exclude $D$.  The aim of this paper is to remove the need to exclude
$K_5 - e$; that is, to prove the Borodin-Kostochka
Conjecture for claw-free graphs. 

\begin{repthm}{BKClawFree}
Every claw-free graph satisfying $\chi \geq \Delta \geq 9$ contains a
$K_\Delta$.
\end{repthm}

This also generalizes the result of Beutelspacher and Hering \cite{beutelspacher1984minimal} that the
Borodin-Kostochka conjecture holds for graphs with independence number at most
two.  The value of $9$ in Theorem \ref{BKClawFree} is best possible since the
graph with $\Delta = 8$ in Figure \ref{fig:SmallCE} is claw-free (both this example and the following tightness example appear in Section 11.2 of \cite{molloyreed2002book}).
Theorem \ref{BKClawFree} is also optimal in the following sense.  We can
reformulate the statement as: every claw-free graph with $\Delta \geq 9$
satisfies $\chi \leq \max\{\omega, \Delta - 1\}$.  Consider a similar statement
with $\Delta - 1$ replaced by $f(\Delta)$ for some $\func{f}{\IN}{\IN}$ and $9$
replaced by $\Delta_0$. We show that $f(x) \geq x - 1$ for $x \geq \Delta_0$. 
Consider $G_t \DefinedAs \join{K_t}{C_5}$ (here $\join{A}{B}$ denotes the {\it
join} of $A$ and $B$ and is formed from $A$ and $B$ by adding all edges with one
endpoint in $A$ and the other in $B$).  We have $\chi(G_t) = t + 3$, $\omega(G_t) = t + 2$,
and $\Delta(G_t) = t + 4$ and $G_t$ is claw-free.  Hence for $t \geq \Delta_0
- 4$ we have $t + 3 \leq \max\set{t + 2, f(t + 4)} \leq f(t + 4)$ giving $f(x)
\geq x - 1$ for $x \geq \Delta_0$.

As shown in \cite{rabern2011strengthening}, the situation is very different for
line graphs of multigraphs, which satisfy $\chi \leq \max\{\omega,
\frac{7\Delta + 10}{8}\}$.  There it was conjectured that $f(x) \DefinedAs
\frac{5\Delta + 8}{6}$ works for line graphs of multigraphs; this would be best
possible.  The example $\join{K_t}{C_5}$ is claw-free, but it is not
quasi-line.

\begin{question}
What is the situation for quasi-line graphs?  That is, what is the optimal
$f$ such that every quasi-line graph with large enough maximum degree satisfies
$\chi \leq \max\{\omega, f(\Delta)\}$.
\end{question}

Borodin and Kostochka conjectured (to themselves) \cite{PersonalComms} that
their conjecture also holds for list coloring. 

\begin{conjecture}[Borodin and Kostochka \cite{PersonalComms}]\label{BKList}
Every graph with $\Delta \geq 9$ satisfies $\chi_l \leq \max\{\omega, \Delta -
1\}$.
\end{conjecture}

We make some progress on this conjecture for claw-free graphs, proving it for
circular interval graphs and severely restricting line graph counterexamples. 
These two classes are the base cases of the structure theorem for quasi-line
graphs of Chudnovsky and Seymour \cite{chudnovsky2005structure} that we use. 
Finally, we prove the following.

\begin{repthm}{ClawFreeLiftForLists}
If every quasi-line graph satisfying $\chi_l \geq \Delta \geq 9$ contains a
$K_\Delta$, then the same statement holds for every claw-free graph.
\end{repthm}

In \cite{gravier1998graphs}, Gravier and Maffray conjecture the following
strengthening of the list coloring conjecture.  Conjecture \ref{BKList}
for claw-free graphs would be an immediate consequence.

\begin{conjecture}[Gravier and Maffray \cite{gravier1998graphs}]
Every claw-free graph satisfies $\chi_l = \chi$.
\end{conjecture}

The outline of this paper is as follows.  A quasi-line graph is one in which the neighborhood of every  vertex can be covered by two cliques.  Quasi-line graphs are  a proper subset of claw-free graphs and a proper superset of line graphs.  We use a structure theorem of Chudnovsky and Seymour, which says (roughly) that every quasi-line graph is either a (i) a line graph, (ii) a circular interval graph, or (iii) the result of ``pasting together'' smaller quasi-line graphs.  So to prove the Borodin-Kostochka conjecture for quasi-line graphs, we prove it for circular interval graphs, we recall Rabern's proof for line graphs, and we show how to ``paste together'' good colorings of smaller graphs to get a good coloring of a larger graph.

If a graph G is claw-free, but not quasi-line, then we show that $G$ contains a vertex $v$ with an induced $C_5$ in its neighborhood.  We use the presence of this induced $\join{K_1}{C_5}$ to show that $G$ must contain a $d_1$-choosable subgraph (defined in Section~\ref{ListLemmas}).  Since such a subgraph cannot appear in a vertex critical graph, this completes the proof of the Borodin-Kostochka conjecture for claw-free graphs.  (In fact, this reduction from claw-free to quasi-line graphs works equally well for the list version of the Borodin-Kostochka Conjecture.)

It is likely that some of our list coloring arguments could be shortened by using Ohba's Conjecture, which was recently proved by Noel, Reed, and Wu~\cite{noel2013Ohbas}.  However, we  prefer to keep this paper as self-contained as possible.  On a related note, by using a lemma of Kostochka \cite{kostochkaRussian}, we can reduce the Borodin-Kostochka Conjecture for any hereditary graph class to the case when $\Delta=9$ (see the introduction of~\cite{mules} for more details).  However, this reduction does not seem to simplify any of our proofs and does not work for list coloring, so we omit it.

Now we introduce some notation and terminology that will be used through the
paper. We write $K_t$ for the complete graph on $t$ vertices and $E_t$ for the edgeless graph on $t$ vertices. If $G$ is a vertex critical graph with $\chi = \Delta$, then every vertex
in $G$ has degree $\Delta - 1$ or $\Delta$.  We call the former vertices
\emph{low} and the latter vertices \emph{high}. For vertices $x,y$ in $G$, we
write $x \adj y$ if $xy \in E(G)$ and $x \nonadj y$ if $xy \not \in E(G)$.  
An {\it almost complete graph}
is a graph $G$ for which there exists $v\in V(G)$ such that $G-v$ is a complete graph. 
We write {\it diamond} for $K_4-e$ and we write {\it paw} for $K_3$ with a pendant edge, that is $\overline{P_3+K_1}$.  We write {\it chair} for the graph formed by subdividing a single edge of $K_{1,3}$.
All the definitions for list coloring that we use are at the start of Section
\ref{ListLemmas}.

\section{List coloring lemmas}\label{ListLemmas}
\subsection{The main idea}
We investigate the structure of vertex critical graphs with $\chi = \Delta$. Let $G$ be such a graph. All of our list coloring lemmas serve the same purpose: exclude graphs from being induced subgraphs of $G$.  To see how this works, let $F$ be an induced subgraph of $G$. Since $G$ is vertex critical, we may $(\Delta - 1)$-color $G-F$.  After doing so, we give each $v \in V(F)$ a list of colors $L(v)$ by taking $\set{1, \ldots, \Delta-1}$ and removing all colors appearing on neighbors of $v$ in $G-F$.  Then, as $v$ has at most $d_G(v) - d_F(v)$ neighbors in $G-F$ we have $\card{L(v)} \geq \Delta-1 - (d_G(v) - d_F(v)) \geq d_F(v) - 1$. If we could properly color $F$ from these lists, we would have a $(\Delta-1)$-coloring of $G$, which is impossible.  We call a graph $F$ $d_1$-choosable if it can be colored from any list assignment $L$ with $\card{L(v)} \geq d_F(v) - 1$ for each $v \in V(F)$.  Then, as we just saw, no $d_1$-choosable graph can be an induced subgraph of $G$.  So, by finding many small $d_1$-choosable graphs, we can severely restrict the structure of $G$.  The next section gives the formal definitions and list coloring lemmas that we need for this application.  The reader should feel free to skip this section for now and return as needed.

\subsection{The details}
Let $G$ be a graph.  A \emph{list assignment} to the vertices of $G$ is a
function from $V(G)$ to the finite subsets of $\mathbb{N}$.  A list assignment
$L$ to $G$ is \emph{good} if $G$ has a proper coloring $c$ where $c(v) \in L(v)$ for
each $v \in V(G)$.  It is \emph{bad} otherwise.  We call the collection of all
colors that appear in $L$, the \emph{pot} of $L$.  That is $Pot(L) \DefinedAs
\bigcup_{v \in V(G)} L(v)$.  For a subset $A$ of $V(G)$ we write $Pot_A(L)
\DefinedAs \bigcup_{v \in A} L(v)$.  Also, for a subgraph $H$ of $G$, put
$Pot_H(L) \DefinedAs Pot_{V(H)}(L)$. For $S \subseteq Pot(L)$, let $G_S$ be the
graph $G\left[\setb{v}{V(G)}{L(v) \cap S \neq \emptyset}\right]$.  We also write $G_c$ for $G_{\{c\}}$. For $\func{f}{V(G)}{\IN}$, an $f$-assignment on $G$ is an assignment $L$ of lists to the vertices of $G$ such that $\card{L(v)} = f(v)$
for each $v \in V(G)$.  We say that $G$ is {\it $f$-choosable} if every
$f$-assignment on $G$ is good.  We call a graph that is $f$-choosable where $f(v) \DefinedAs d(v) - 1$ a $d_1$-choosable graph. 

We restate some of the results on $d_1$-choosable graphs from \cite{mules} that we need here; we omit all of their proofs.
We do prove Lemma \ref{ComponentsOfColor} which is a strengthening of a special case of Lemma~\ref{CannotColorSelfWithSelf} (and which is not proved in \cite{mules}).

We need the following list coloring lemmas from \cite{mules}.  
Given a graph $G$ and
$\func{f}{V(G)}{\mathbb{N}}$, we have a partial order on the $f$-assignments to $G$ given by $L < L'$ iff $\card{Pot(L)} < \card{Pot(L')}$.  
When we talk of \emph{minimal} $f$-assignments, we mean minimal with respect to this partial order.

\begin{SmallPotLemma}
Let $G$ be a graph and $\func{f}{V(G)}{\mathbb{N}}$ with $f(v) < \card{G}$ for all $v \in V(G)$.  
If $G$ is not $f$-choosable, then $G$ has a bad $f$-assignment $L$ such that $\card{Pot(L)} < \card{G}$.
\end{SmallPotLemma}

The core of the Small Pot Lemma is the following.  We will also prove a lemma that
gets more when $|S| = 1$.

\begin{lem}\label{CannotColorSelfWithSelf}
Let $G$ be a graph and $\func{f}{V(G)}{\mathbb{N}}$.  
Suppose $G$ is not $f$-choosable and let $L$ be a minimal bad $f$-assignment.
Assume $L(v) \neq Pot(L)$ for each $v \in V(G)$.  Then, for each nonempty $S \subseteq Pot(L)$, any coloring of $G_S$ from $L$ uses some color not in $S$.
\end{lem}

When $|S| = 1$, we can say more.  We use the following lemma in the proof
that the graph $D_8$ in Figure \ref{fig:D8} is $d_1$-choosable.  It should be
useful elsewhere as well.

\begin{lem}\label{ComponentsOfColor}
Let $G$ be a graph and $\func{f}{V(G)}{\mathbb{N}}$.  
Suppose $G$ is not $f$-choosable and let $L$ be a minimal bad $f$-assignment.
Then for any $c \in Pot(L)$, there is a component $H$ of $G_c$ such that
$Pot_H(L) = Pot(L)$.  In particular, $Pot_{G_c}(L) = Pot(L)$.
\end{lem}
\begin{proof}
Suppose otherwise that we have $c \in Pot(L)$ such that $Pot_H(L) \subsetneq
Pot(L)$ for all components $H$ of $G_c$.  Say the components of $G_c$ are
$H_1, \ldots, H_t$. For $i \in \irange{t}$, choose $\alpha_i \in Pot(L) -
Pot_{H_i}(L)$. Now define a list assignment $L'$ on $G$ by setting $L'(v)
\DefinedAs L(v)$ for all $v \in V(G) - V(G_c)$ and for each $i \in \irange{t}$
setting $L'(v) \DefinedAs (L(v) - c) \cup \set{\alpha_i}$ for each $v \in
V(H_i)$.  Then $\card{Pot(L')} < \card{Pot(L)}$ and hence by minimality $L$ we
have an $L'$-coloring $\pi$ of $G$.  Plainly $Q \DefinedAs
\setb{v}{V(G_c)}{\pi(v) = \alpha_i \text{ for some $i \in \irange{t}$}}$ is an
independent set.  Since $c$ does not appear outside $G_c$, we can recolor
all vertices in $Q$ with $c$ to get an $L$-coloring of $G$.  This contradicts
the fact that $L$ is bad.
\end{proof}

The next two lemmas allow us to color pairs in $H$ without worrying about
completing the coloring to $H$.

\begin{lem}\label{NeighborhoodPotShrink}
Let $H$ be a $d_0$-choosable graph such that $G \DefinedAs \join{K_1}{H}$ is not
$d_1$-choosable and $L$ a minimal bad $d_1$-assignment on $G$.  If some
nonadjacent pair in $H$ have intersecting lists, then $\card{Pot(L)} \leq \card{H} - 1$.
\end{lem}

With the same proof, we have the following.

\begin{lem}\label{LowSinglePair}
Let $H$ be a $d_0$-choosable graph such that $G \DefinedAs \join{K_1}{H}$ is not
$f$-choosable where $f(v) \geq d(v)$ for the $v$ in the $K_1$ and $f(v) \geq
d(x) - 1$ for $x \in V(H)$. If $L$ is a minimal bad $f$-assignment on $G$, then
all nonadjacent pairs in $H$ have disjoint lists.
\end{lem}

\begin{lem}\label{ConnectedAtLeast4Poss}
Let $A$ be a connected graph with $\card{A} \geq 4$ and $B$ an arbitrary graph. If $\join{A}{B}$ is not $d_1$-choosable, 
then $B$ is $\join{E_3}{K_{\card{B} - 3}}$ or almost complete.
\end{lem}

\begin{lem}\label{ConnectedEqual3Poss}
\label{K3Classification}
$\join{K_3}{B}$ is not $d_1$-choosable iff
$B$ is one of the following: almost complete, $\djunion{K_t}{K_{\card{B} - t}}$,
$\djunion{\djunion{K_1}{K_t}}{K_{\card{B} - t - 1}}$, $\djunion{E_3}{K_{\card{B}
- 3}}$, or $\card{B} \leq 5$ and $B = \join{E_3}{K_{\card{B} - 3}}$.
\end{lem}

\begin{lem}\label{K2Classification}
If $K_2*B$ is not $d_1$-choosable, then $B$ consists of a disjoint union of
complete subgraphs, together with at most one incomplete component $H$.  If $H$
has a dominating vertex $v$, then $K_2*H = K_3*(H-v)$, so by
Lemma~\ref{ConnectedEqual3Poss} we can
completely describe $H$.  Otherwise $H$ is formed either by adding an edge
between two disjoint cliques or by adding a single pendant edge incident to 
each of two distinct vertices of a clique.
Furthermore, all graphs formed in this way are not $d_1$-choosable.
\end{lem}

Pulling out some particular cases makes for easier application.  A {\it chair}
is formed from $K_{1,3}$ by subdividing an edge. An {\it antichair} is the complement of a chair.

\begin{lem}\label{K2Antichair}
$\join{K_2}{\text{antichair}}$ is $d_1$-choosable.
\end{lem}

\begin{lem}\label{K3P4}
$\join{K_3}{P_4}$ is $d_1$-choosable.
\end{lem}

The situation is simpler for joins with $E_2$, as shown by the next lemma.

\begin{lem}\label{E2JoinB}
$\join{E_2}{B}$ is not $d_1$-choosable iff $B$ is the disjoint union of complete subgraphs and at most one copy of $P_3$.
\end{lem}

We often need to handle low vertices in our proofs, which corresponds to a vertex
with $\card{L(v)} \geq d(v)$ when we try to complete the partial coloring.

\begin{lem}\label{mixed}
Let $A$ be a graph with $\card{A} \geq 4$.  
Let $L$ be a list assignment on $G \DefinedAs \join{E_2}{A}$ such that $\card{L(v)} \geq d(v) - 1$ for all $v \in V(G)$ and each component $D$ of $A$ has 
a vertex $v$ such that $\card{L(v)} \geq d(v)$.  Then $L$ is good on $G$.
\end{lem}

\begin{lem}\label{mixed3}
Let $A$ be a graph with $\card{A} \geq 3$.  Let $L$ be a list assignment on $G \DefinedAs \join{E_2}{A}$ such that $\card{L(v)} \geq d(v) - 1$ for all $v \in V(G)$, $\card{L(v)} \geq d(v)$ for some $v$ in the $E_2$ and each component $D$ of $A$ has a vertex $v$ such that $\card{L(v)} \geq d(v)$.  Then $L$ is good on $G$.
\end{lem}

\begin{lem}\label{IntersectionsInB}
Let $H$ be a $d_0$-choosable graph such that $G \DefinedAs \join{K_1}{H}$ is not $d_1$-choosable; let $L$ be a bad $d_1$-assignment on $G$.  Then
\begin{enumerate}
\item for any independent set $I \subseteq V(H)$ with $\card{I} = 3$, we have
$\bigcap_{v \in I} L(v) = \emptyset$;
\item for disjoint nonadjacent pairs $\set{x_1, y_1}$ and $\set{x_2, y_2}$ at least one of the following holds
	\begin{enumerate}
	\item $L(x_1) \cap L(y_1) = \emptyset$;
	\item $L(x_2) \cap L(y_2) = \emptyset$;
	\item $\card{L(x_1) \cap L(y_1)} = 1$ and $L(x_1) \cap L(y_1) = L(x_2) \cap L(y_2)$.
	\end{enumerate}
\end{enumerate}
\end{lem}

Let $E_2^n$ denote the join of $n$ copies of $E_2$, i.e., $E_2^n$ is isomorphic
to $K_{2n}-E(M)$, where $M$ is a perfect matching.  
The following lemma first appeared in~\cite{erdos1979choosability}.
We also prove it in \cite{mules}. 

\begin{lem}\label{E2n}
$E_2^n$ is $n$-choosable.
\end{lem}

\section{Circular interval graphs}
%
%

Given a set $V$ of points on the unit circle together with a set of closed intervals $C$ on the unit circle we define a graph with vertex set $V$ and an edge between two different vertices if and only if they are both contained in some element of $C$.  Any graph isomorphic to such a graph is a \emph{circular interval graph}.  Similarly, by replacing the unit circle with the unit interval, we get the class of \emph{linear interval graphs}.

\begin{lem}\label{CircularIntervalLemma}
Every circular interval graph satisfying $\chi_l \ge \Delta \ge 9$ contains a
$K_\Delta$.
\end{lem}
\begin{proof}
Suppose the contrary and choose a counterexample $G$ minimizing $\card{G}$.  Put
$\Delta \DefinedAs \Delta(G)$. Then $\chi_l(G)=\Delta$, $\omega(G)\le \Delta-1$,
$\delta(G)\ge \Delta-1$ and $\chi_l(G-v)\le \Delta-1$ for all $v\in V(G)$. 
Since $G$ is a circular interval graph, by definition $G$ has a representation
in a cycle $v_1v_2\ldots v_n$.  Let $K$ be a maximum clique in $G$.  By symmetry we may assume that
$V(K)=\{v_1,v_2,\ldots,v_t\}$ for some $t\le \Delta-1$; further, if possible we
label the vertices so that $v_{t-3}\adj v_{t+1}$ and the edge goes through
$v_{t-2},v_{t-1},v_t$.

\claim{1} {$v_1\nonadj v_{t+1}$ and $v_2\nonadj v_{t+2}$ and
$v_1\nonadj v_{t+2}$.}
Assume the contrary.
Clearly we cannot have $v_1\adj v_{t+1}$ and have the edge go through
$v_2,v_3,\ldots, v_t$ (since then we get a clique of size $t+1$).
Similarly, we cannot have $v_2\adj v_{t+2}$ and have the edge go through
$v_3,v_4,\ldots,v_{t+1}$.  So assume the edge $v_1v_{t+2}$ exists and 
goes around the other way.
If $v_1\adj v_{t+1}$, then let $G'=G\setminus \{v_1\}$ and if $v_1\nonadj
v_{t+1}$, then let $G'=G\setminus \{v_1,v_{t+1}\}$.  Now let
$V_1=\{v_2,v_3,\ldots,v_t\}$ and $V_2=V(G')\setminus V_1$.  Let $K'=G[V_1]$ and
$L'=G[V_2]$; note that $K'$ and $L'$ are each cliques of size at most
$\Delta-2$.  Now for each $S\subseteq V_2$,
we have $|N_{\overline{G}}(S)\cap V_1|\ge |S|$
(otherwise we get a clique of size $t$ in $G'$ and a clique of size $t+1$ in $G$).  Now by Hall's Theorem, we have a matching in $\overline{G}$ between $V_1$
and $V_2$ that saturates $V_2$.  This implies that $G'\subseteq E_2^{\Delta-2}$,
which in turn gives $G\subseteq E_2^{\Delta-1}$.  By Lemma~\ref{E2n}, $G$ is
$(\Delta-1)$-choosable, which is a contradiction.

\claim{2} {$v_{t-3}\adj v_{t+1}$ and the edge passes through
$v_{t-2},v_{t-1},v_t$.}
Assume the contrary.
If $t\ge 7$, then since $t\le \Delta-1$, $v_4$ has some
neighbor outside of $K$; by (reflectional) symmetry we could have labeled the
vertices so that $v_{t-3}\adj v_{t+1}$.
So we must have $t\le 6$.
Each vertex $v$ that is high has either at least $\ceil{\Delta/2}$ clockwise
neighbors or at least $\ceil{\Delta/2}$ counterclockwise neighbors.  This gives a clique
of size $1+\ceil{\Delta/2}\ge 6$.  If $v_3$ is high, then either $v_3$ has at
least 4 clockwise neighbors, so $v_3\adj v_7$, or else $v_3$ has at least 6
counterclockwise neighbors, so $|K|\ge 7$.  Thus, we may assume that $v_3$ is
low; by symmetry (and our choice of labeling prior to Claim~1) $v_4$ is also
low.  Now since $v_4$ has only 3 counterclockwise neighbors, we get $v_4\adj
v_7$ (in fact, we get $v_4\adj v_9$).  Thus, $\{v_3,v_4,v_5,v_6,v_7\}$ induces
$K_3*E_2$ with a low degree vertex in both the $K_3$ and the $E_2$, which 
contradicts Lemma~\ref{mixed3}.

\claim{3} {$v_{t-2} \nonadj v_{t+2}$.}
Assume the contrary.  By
Claim~1 the edge goes through $v_{t-1},v_t,v_{t+1}$.  If $v_{t-3}\adj v_{t+2}$,
then $\{v_1,v_2,v_{t-3},v_{t-2},v_{t-1},v_t,v_{t+1},v_{t+2}\}$ induces $K_4*B$,
where $B$ is not almost complete; this contradicts
Lemma~\ref{ConnectedAtLeast4Poss}.  If $v_{t-3}\nonadj v_{t+2}$, then we get a
$K_3*P_4$ induced by $\{v_1, v_{t-3}, v_{t-2}, v_{t-1}, v_t, v_{t+1},
v_{t+2}\}$, which contradicts Lemma~\ref{K3P4}. 

\claim{4} {$v_{t-1}\nonadj v_{t+2}$.}
Suppose the contrary.  Now
$\{v_1,v_{t-3},v_{t-2},v_{t-1},v_t,v_{t+1},v_{t+2}\}$ induces
$K_2*\mbox{antichair}$ (with $v_{t-1},v_t$ in the $K_2$), which contradicts
Lemma~\ref{K2Antichair}.

\claim{5} {$G$ is $(\Delta-1)$-choosable.}
Let $S=\{v_{t-3},v_{t-2},v_{t-1},v_t\}$.  If any vertex of $S$ is low, then
$S\cup\{v_1,v_{t+1}\}$ induces $K_4*E_2$ with a low vertex in the $K_4$, which
contradicts Lemma~\ref{mixed}.  So all of $S$ is high.  If
$v_t\nonadj v_{t+2}$, then $\{v_t,v_{t-1},\ldots, v_{t-\Delta+1}\}$ (subscripts
are modulo $n$) induces $K_{\Delta}$.  So $v_t\adj v_{t+2}$.  Since
$v_{t-1}\nonadj v_{t+2}$ and all of $S$ is high, we get $v_n\in (\cap_{v\in
(S\setminus\{v_t\})}N(v))\setminus N(v_t)$.  Now we must have $v_n\nonadj
v_{t+1}$ (for otherwise $G$ is $(\Delta-1)$-choosable, as in Claim~1).  So we
get $K_3*P_4$ induced by $\{v_{t+1},v_t,v_{t-1},v_{t-2},v_{t-3},v_1,v_n\}$,
which contradicts Lemma~\ref{K3P4}.  
\end{proof}

\section{Quasi-line graphs}
A graph is \emph{quasi-line} if every vertex is bisimplicial (its neighborhood can be covered by two cliques).  
We apply a version of Chudnovsky and Seymour's structure theorem for quasi-line
graphs from King's thesis \cite{king2009claw}. The undefined terms
will be defined after the statement.

\begin{lem}\label{QuasilineStructure}
Every connected skeletal quasi-line graph is a circular interval graph or a composition of
linear interval strips.
\end{lem}

A \emph{homogeneous pair of cliques} $(A_1, A_2)$ in a graph $G$ is a pair of
disjoint nonempty cliques such that for each $i \in \irange{2}$, every vertex in
$G - (A_1 \cup A_2)$ is either joined to $A_i$ or misses all of $A_i$ and
$\card{A_1} + \card{A_2} \geq 3$. A homogeneous pair of cliques $(A_1, A_2)$ is $\emph{skeletal}$
if for any $e \in E(A, B)$ we have $\omega(G[A \cup B] - e) < \omega(G[A \cup
B])$.  A graph is $\emph{skeletal}$ if it contains no nonskeletal homogeneous
pair of cliques.

Generalizing a lemma of Chudnovsky and Fradkin \cite{chudnovskyFradkin}, King
proved a lemma allowing us to handle nonskeletal homogeneous pairs
of cliques.

\begin{lem}[King \cite{king2009claw}]\label{NoHomogeneous} If $G$ is a
nonskeletal graph, then there is a proper subgraph $G'$ of $G$ such that:
\begin{enumerate}
  \item $G'$ is skeletal;
  \item $\chi(G') = \chi(G)$;
  \item If $G$ is claw-free, then so is $G'$;
  \item If $G$ is quasi-line, then so is $G'$.
\end{enumerate}
\end{lem}

It remains to define the generalization of line graphs introduced by Chudnovsky
and Seymour \cite{chudnovsky2005structure}; this is the notion of
\emph{compositions of strips} (for a more detailed introduction, see Chapter 5
of~\cite{king2009claw}). We use the modified definition from King and
Reed \cite{king2008bounding}. A \emph{strip} $(H, A_1, A_2)$ is a claw-free
graph $H$ containing two cliques $A_1$ and $A_2$ such that for each $i \in \irange{2}$ and $v \in A_i$, $N_H(v) - A_i$ is a clique.  
If $H$ is a linear interval graph with $A_1$ and $A_2$ on opposite ends, then $(H, A_1, A_2)$ is a $\emph{linear
interval strip}$.  Now let $H$ be a directed multigraph (possibly with loops)
and suppose for each edge $e$ of $H$ we have a strip $(H_e, X_e, Y_e)$.  For
each $v \in V(H)$ define

\[C_v \DefinedAs \parens{\bigcup \setbs{X_e}{\text{$e$ is directed out of $v$}}}
\cup \parens{\bigcup \setbs{Y_e}{\text{$e$ is directed into $v$}}}\]

The graph formed by taking the disjoint union of $\setbs{H_e}{e \in E(H)}$ and
making $C_v$ a clique for each $v \in V(H)$ is the composition of the strips
$(H_e, X_e, Y_e)$.  Any graph formed in such a manner is called a
\emph{composition of strips}.  It is easy to see that if for
each strip $(H_e, X_e, Y_e)$ in the composition we have $V(H_e) = X_e = Y_e$,
then the constructed graph is just the line graph of the multigraph formed by
replacing each $e \in E(H)$ with $\card{H_e}$ copies of $e$.

It will be convenient to have notation and terminology for a strip together with
how it attaches to the graph. An \emph{interval $2$-join} in a graph $G$ is an
induced subgraph $H$ such that:

\begin{enumerate}
\item $H$ is a (nonempty) linear interval graph,
\item The ends of $H$ are (not necessarily disjoint) cliques $A_1$, $A_2$,
\item $G-H$ contains cliques $B_1$, $B_2$ (not necessarily disjoint) such that $A_1$ is joined to $B_1$ and $A_2$ is joined to $B_2$,
\item there are no other edges between $H$ and $G-H$.
\end{enumerate}

Note that $A_1, A_2, B_1, B_2$ are uniquely determined by $H$, so we
are justified in calling both $H$ and the quintuple $(H, A_1, A_2, B_1, B_2)$
the interval $2$-join. An interval $2$-join $(H, A_1, A_2, B_1, B_2)$ is
\emph{trivial} if $V(H) = A_1 = A_2$ and \emph{canonical} if $A_1 \cap A_2 =
\emptyset$.  A canonical interval $2$-join $(H, A_1, A_2, B_1, B_2)$ with
leftmost vertex $v_1$ and rightmost vertex $v_t$ is \emph{reducible} if $H$ is
incomplete and $N_H(A_1)\setminus A_1 = N_H(v_1)\setminus A_1$ or
$N_H(A_2)\setminus A_2 = N_H(v_t)\setminus A_2$.  We call such a canonical
interval $2$-join reducible because we can \emph{reduce} it as follows.  Suppose
$H$ is incomplete and $N_H(A_1)\setminus A_1 = N_H(v_1)\setminus A_1$. Put $C
\DefinedAs N_H(v_1) \setminus A_1$ and then $A_1' \DefinedAs C \setminus A_2$ and $A_2' \DefinedAs A_2 \setminus C$.  
Since $H$ is not complete $v_t \in A_2'$ and hence $H' \DefinedAs G[A_1' \cup
A_2']$ is a nonempty linear interval graph that gives the reduced canonical
interval $2$-join $(H', A_1', A_2', A_1 \cup \parens{C \cap A_2}, B_2 \cup
\parens{C \cap A_2})$.

\begin{lem}\label{Irreducible2Join}
If $(H, A_1, A_2, B_1, B_2)$ is an irreducible canonical interval $2$-join in a
skeletal vertex critical graph $G$ with $\chi(G) = \Delta(G) \geq 9$, then $B_1 \cap B_2
= \emptyset$, $\card{A_1}, \card{A_2} \leq 3$ and $H$ is complete.
\end{lem}
\begin{proof}
Let $(H, A_1, A_2, B_1, B_2)$ be an irreducible canonical interval $2$-join in a
skeletal vertex critical graph $G$ with $\chi(G) = \Delta(G) \geq 9$.  Put $\Delta
\DefinedAs \Delta(G)$.

Note that, since it is vertex critical, $G$
contains no $K_\Delta$ and in particular $G$ has no simplicial vertices.  Label
the vertices of $H$ left-to-right as $v_1, \ldots, v_t$.  Say $A_1 = \set{v_1, \ldots, v_L}$ and $A_2 = \set{v_R, \ldots, v_t}$. For $v \in V(H)$, define $r(v) \DefinedAs
\max\setbs{i \in \irange{t}}{v \adj v_i}$ and $l(v) \DefinedAs \min\setbs{i \in \irange{t}}{v \adj v_i}$.  
These are well-defined since $\card{H} \geq 2$ and $H$ is connected by the following claim.

\claim{1} {$A_1, A_2, B_1, B_2 \neq \emptyset$, $B_1 \not \subseteq
B_2$, $B_2 \not \subseteq B_1$ and $H$ is connected.} 
Otherwise $G$ would contain a clique cutset.

\claim{2} {If $H$ is complete, then $R - L = 1$.}  Suppose $V(H) \neq
A_1 \cup A_2$. Then any $v \in V(H) \setminus A_1 \cup A_2$ would be simplicial
in $G$, which is impossible.  Hence $R - L = 1$.

\claim{3} {If $H$ is not complete, then $r(v_L) = r(v_1) + 1$ and
$l(v_R) = l(v_t) - 1$.  In particular, $v_1, v_t$ are low and $\card{A_1},
\card{A_2} \geq 2$.} Suppose otherwise that $H$ is not
complete and $r(v_L) \neq r(v_1) + 1$. By definition, $N_H(v_1) \subseteq N_H(v_L)$ and $v_1, v_L$ have the same neighbors
in $G\setminus H$.  Hence if $r(v_L) > r(v_1) + 1$, then $d(v_L) - d(v_1) \geq
2$, impossible.  So we must have $r(v_L) = r(v_1)$ and hence $N_H(A_1)\setminus
A_1 = N_H(v_1)\setminus A_1$.  Thus the $2$-join is reducible, a contradiction.
Therefore $r(v_L) = r(v_1) + 1$.  Similarly, $l(v_R) = l(v_t) - 1$. 

\claim{4} {$\card{A_1}, \card{A_2} \leq 3$.}  Suppose otherwise that
$\card{A_1} \geq 4$.  First, suppose $H$ is complete.  By Claim~2, $V(H) = A_1
\cup A_2$. If $v_1$ is low, then for any $w_1 \in
B_1 \setminus B_2$ the vertex set $\{v_1, \ldots, v_4, v_t, w_1\}$ induces
a $\join{K_4}{E_2}$ violating Lemma \ref{mixed}.  Hence $v_1$ is high. If
$\card{A_2} \geq 2$ and $\card{B_1 \setminus B_2} \geq 2$, 
then for any $w_1, w_2 \in B_1 \setminus B_2$, the vertex set $\{v_1, \ldots, v_4, v_{t-1}, v_t, w_1, w_2\}$ induces
a $\join{K_4}{2K_2}$, which is impossible by Lemma \ref{ConnectedAtLeast4Poss}. 
Hence either $\card{A_2} = 1$ or $\card{B_1 \setminus B_2} = 1$.  Suppose
$\card{A_2} = 1$.  Then, since $A_1 \cup B_1$ induces a clique and $\card{A_1
\cup B_1} = d(v_1)$, $v_1$ must be low, impossible.    Hence we must have
$\card{B_1 \setminus B_2} = 1$.  Thus $\card{B_1 \cap B_2} = \card{B_1} - 1$. 
Hence $V(H) \cup B_1 \cap B_2$ induces a clique with $\card{A_1} + \card{A_2} +
\card{B_1} - 1 = d(v_1) = \Delta$ vertices, impossible.

Therefore $H$ must be incomplete.  By Claim~3, $v_1$ is low.  But then as above
for any $w_1 \in B_1 \setminus B_2$ the vertex set $\{v_1, \ldots, v_4, v_{L+1}, w_1\}$ induces
a $\join{K_4}{E_2}$ violating Lemma \ref{mixed}.  Hence we must have $\card{A_1}
\leq 3$.  Similarly, $\card{A_2} \leq 3$.

\claim{5} {$R - L = 1$.}  Suppose otherwise that $R - L \geq 2$.  Then
by Claim~2, $H$ is incomplete.  Hence by Claim~3, $r(v_L) = r(v_1) + 1$,
$l(v_R) = l(v_t) - 1$, $v_1, v_t$ are low and $\card{A_1}, \card{A_2} \geq 2$.  
But then $\parens{A_1, \set{v_{r(v_L)}}}$ is a non-skeletal homogeneous pair of cliques in $G$, impossible.

\claim{6} {$B_1 \cap B_2 = \emptyset$.} Suppose otherwise that we have
$w \in B_1 \cap B_2$. 

\subclaim{6a} {Each $v \in V(H)$ is low, $\card{B_1} = \card{B_2}$,
$\card{B_1 \setminus B_2} = \card{B_2 \setminus B_1} = 1$, $d(v) = \card{A_1} + \card{A_2} +
\card{B_1} - 1$ for each $v \in V(H)$ and $H$ is complete.} By Claim~5, we have
$d(v) \leq \card{A_1} + \card{A_2} + \card{B_1} - 1$ for $v \in A_1$ and $d(v) \leq \card{A_1} + \card{A_2} +
\card{B_2} - 1$ for $v \in A_2$.  Also, as $B_1 \not \subseteq B_2$ and $B_2
\not \subseteq B_1$, we have $d(w) \geq \max\set{\card{B_1}, \card{B_2}} +
\card{A_1} + \card{A_2}$.  So $d(w) \geq d(v) + 1$ for any $v \in V(H)$.  This
implies that each $v \in V(H)$ is low, $\card{B_1} = \card{B_2}$, $\card{B_1
\setminus B_2} = \card{B_2 \setminus B_1} = 1$, $d(v) = \card{A_1} + \card{A_2} +
\card{B_1} - 1$ for each $v \in V(H)$ and hence $H$ is complete.  

\subclaim{6b} {$\card{B_1 \cap B_2} \leq 3$.} Suppose otherwise that
$\card{B_1 \cap B_2} \geq 4$.  Pick $w_1 \in B_1 \setminus B_2$, $w_2 \in B_2
\setminus B_1$ and $z_1, z_2, z_3, z_4 \in B_1 \cap B_2$.  Then the set
$\set{z_1, z_2, z_3, z_4, w_1, w_2, v_1, v_t}$ induces a subgraph violating Lemma
\ref{ConnectedAtLeast4Poss}.  Hence $\card{B_1 \cap B_2} \leq 3$.

\subclaim{6c} {Claim~6 is true.}  By Subclaim~6a and Subclaim~6b we
have $3 \geq \card{B_1 \cap B_2} = \card{B_1} - 1$ and hence $\card{B_1} =
\card{B_2} \leq 4$.  Suppose $\card{A_1}, \card{A_2} \leq 2$.  Then $\Delta - 1
= d(v_1) \leq 3 + \card{B_1} \leq 7$, a contradiction.  Hence by symmetry we may assume
that $\card{A_1} \geq 3$.  But then for $w_1 \in B_1 \setminus B_2$, the
set $\set{v_1, v_2, v_3, v_t, w_1}$ induces a $\join{K_3}{E_2}$ violating Lemma
\ref{mixed3}.

\claim{7} {$H$ is complete.}  Suppose $H$ is incomplete.  By Claim~5, $R - L = 1$. Then, by Claim~3
$r(v_L) = r(v_1) + 1$ and $l(v_R) = l(v_t) - 1$.  Since $v_1$ is not simplicial,
$r(v_1) \geq L + 1 = R$.  Hence $l(v_R) = 1$ and thus $l(v_t) = 2$.  Similarly,
$r(v_1) = t - 1$.  So, $H$ is $K_t$ less an edge.  But $(A_1, A_2)$ is a
homogeneous pair of cliques with $\card{A_1}, \card{A_2} \geq 2$ and hence there
is an edge between $A_1$ and $A_2$ that we can remove without decreasing
$\omega(G[A_1 \cup A_2])$.  This contradicts our assumption that $G$ is skeletal.
\end{proof}

\begin{lem}\label{TrivialOrCanonical}
An interval $2$-join in a skeletal vertex critical graph satisfying $\chi = \Delta \geq
9$ is either trivial or canonical.
\end{lem}  
\begin{proof}
Let $(H, A_1, A_2, B_1, B_2)$ be an interval $2$-join in a skeletal vertex critical graph
satisfying $\chi = \Delta \geq 9$. Suppose $H$ is nontrivial; that is, $A_1 \neq
A_2$.  Put $C \DefinedAs A_1 \cap A_2$. Then $(H \setminus C, A_1 \setminus C, A_2 \setminus C, C \cup B_1, C \cup B_2)$ is a canonical interval $2$-join.  
Reduce this $2$-join until we get an irreducible canonical interval $2$-join $(H', A_1', A_2', B_1', B_2')$ with $H' \unlhd H \setminus C$.  
Since $C$ is joined to $H-C$, it is also joined to $H'$.  Hence $C \subseteq B_1' \cap B_2' = \emptyset$ by Lemma \ref{Irreducible2Join}.  
Hence $A_1 \cap A_2 = C = \emptyset$ showing that $H$ is canonical.
\end{proof}

\begin{thm}\label{QuasiLineColoring}
Every quasi-line graph satisfying $\chi \geq \Delta \geq 9$ contains a
$K_\Delta$.
\end{thm}
\begin{proof}
We will prove the theorem by reducing to the case of line graphs, i.e., for
every strip $(H, A_1, A_2)$ we have $A_1=A_2$.
Suppose not and choose a counterexample $G$ minimizing $\card{G}$.  Plainly, $G$ is vertex critical.  
By Lemma \ref{NoHomogeneous}, we may assume that $G$ is skeletal.  By Lemma \ref{CircularIntervalLemma}, $G$ is not a circular interval graph.  
Therefore, by Lemma \ref{QuasilineStructure}, $G$ is a composition of linear interval strips.  Choose such a composition representation of $G$ using the maximum number of strips.  

Let $(H, A_1, A_2)$ be a strip in the composition.  Suppose $A_1 \neq A_2$.  Put $B_1 \DefinedAs N_{G\setminus H}(A_1)$ and $B_2 \DefinedAs N_{G\setminus H}(A_2)$.  
Then $(H, A_1, A_2, B_1, B_2)$ is an interval $2$-join.  Since $A_1 \neq A_2$,
$H$ is canonical by Lemma \ref{TrivialOrCanonical}. Suppose $H$ is reducible. 
By symmetry, we may assume that $N_H(A_1) \setminus A_1 = N_H(v_1) \setminus A_1$. But then replacing the strip $(H, A_1, A_2)$ with the two strips $(G[A_1], A_1, A_1)$ and $(H\setminus A_1, N_H(A_1)\setminus A_1, A_2)$ gives a composition representation of $G$ using more strips, a contradiction.  
Hence $H$ is irreducible.  By Lemma \ref{Irreducible2Join}, $H$ is complete and thus replacing the strip $(H, A_1, A_2)$ with the two strips $(G[A_1], A_1, A_1)$ and $(G[A_2], A_2, A_2)$ gives another contradiction.

Therefore, for every strip $(H, A_1, A_2)$ in the composition we must have $V(H) = A_1 = A_2$.  Hence $G$ is a line graph of a multigraph.  
But this is impossible by Lemma \ref{BKLineGraph}.
\end{proof}

\section{Claw-free graphs}
In this section we reduce the Borodin-Kostochka conjecture for claw-free graphs
to the case of quasi-line graphs.  We first show that a certain graph cannot appear in the neighborhood of
any vertex in our counterexample.

\begin{figure}[htb]
\centering
\begin{tikzpicture}[scale = 10]
\tikzstyle{VertexStyle}=[shape = circle,	
								 minimum size = 1pt,
								 inner sep = 1.2pt,
                         draw]
\Vertex[x = 0.491999983787537, y = 0.856000006198883, L = \tiny {$x_1$}]{v0}
\Vertex[x = 0.586000025272369, y = 0.761999994516373, L = \tiny {$x_2$}]{v1}
\Vertex[x = 0.389999985694885, y = 0.758000001311302, L = \tiny {$x_5$}]{v2}
\Vertex[x = 0.420000016689301, y = 0.631999999284744, L = \tiny {$x_4$}]{v3}
\Vertex[x = 0.555999994277954, y = 0.629999965429306, L = \tiny {$x_3$}]{v4}
\Vertex[x = 0.811999976634979, y = 0.765999987721443, L = \tiny {$y$}]{v5}
\Edge[](v1)(v0)
\Edge[](v1)(v4)
\Edge[](v4)(v3)
\Edge[](v2)(v3)
\Edge[](v0)(v2)
\Edge[](v5)(v0)
\Edge[](v5)(v1)
\Edge[](v5)(v4)
\Edge[](v5)(v3)
\end{tikzpicture}
\caption{The graph $N_6$.}
\label{fig:N6}
\end{figure}
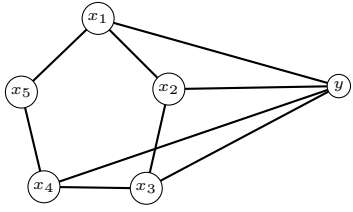

\begin{lem}\label{N6Choosable}
The graph $\join{K_1}{N_6}$ where $N_6$ is the graph in Figure \ref{fig:N6} is
$d_1$-choosable.
\end{lem}
\begin{proof}
Suppose not and let $L$ be a minimal bad $d_1$-assignment on $\join{K_1}{N_6}$. 
Then, by the Small Pot Lemma, $\card{Pot(L)} \leq 6$.  Let $v$ be the vertex in
the $K_1$.  Note that $|L(v)|=5$, $|L(y)|=4$, $|L(x_5)|=2$, and $|L(x_i)|=3$ for
all $i\in[4]$.  Since $\sum_{i=1}^5|L(x_i)| = 14 > |Pot(L)|\omega(C_5)$, we see
that two nonadjacent $x_i$'s have a common color.  Hence, by Lemma
\ref{NeighborhoodPotShrink}, we have $\card{Pot(L)} \leq 5$. Thus we have $c
\in L(y) \cap L(x_5)$.  Also, $L(x_1) \cap L(x_4) \neq \emptyset$, $L(x_1) \cap
L(x_3) \neq \emptyset$ and $L(x_2) \cap L(x_4) \neq \emptyset$.  By Lemma
\ref{IntersectionsInB}, the common color in all of these sets must be $c$. 
Hence $c$ is in all the lists.
Now consider the list assignment $L'$ where $L'(z) = L(z) - c$ for all $z \in
N_6$.  Then $\card{Pot(L')} = 4$ and since $\sum_{i=1}^5|L'(x_i)| = 9 >
|Pot(L')|\omega(C_5)$, we see that that nonadjacent $x_i$'s have a common
color different than $c$.  Now appling Lemma \ref{IntersectionsInB} gives a
final contradiction.
\end{proof}

By a \emph{thickening} of a graph $G$, we just mean a graph formed by replacing
each $x \in V(G)$ by a complete graph $T_x$ such that $\card{T_x} \geq 1$ and
for $x,y \in V(G)$, $T_x$ is joined to $T_y$ iff $x \adj y$.  Each such $T_x$ is
called a \emph{thickening clique}.

\begin{lem}\label{BisimplicialOrThickC5}
Any graph $H$ with $\alpha(H) \leq 2$ such that every
induced subgraph of $\join{K_1}{H}$ is not $d_1$-choosable can either be covered
by two cliques or is a thickening of $C_5$.
\end{lem}
\begin{proof}
Suppose not and let $H$ be a counterexample.  Now by Lemma~\ref{N6Choosable},
$H$ does not contain an induced $N_6$.


\claim{1} {$H$ contains an induced $C_4$ or an induced $C_5$.}
Suppose not.  Then $H$ must be chordal since $\alpha(H) \leq 2$.  In particular,
$H$ contains a simplicial vertex $x$.  But then $\set{x} \cup N_H(x)$ and $V(H)
- N_H(x) - \set{x}$ are two cliques covering $H$, a contradiction.

\claim{2} {$H$ does not contain an induced $C_5$ together with a vertex joined to at least
$4$ vertices in the $C_5$.}
Suppose the contrary.  If the vertex is joined to all of the $C_5$, then we have in $\join{K_1}{H}$
an induced $\join{K_2}{C_5}$, which is $d_1$-choosable by
Lemma~\ref{K2Classification}. If the vertex is joined to only four vertices in
the $C_5$, then $\join{K_1}{H}$ contains an induced $\join{K_1}{N_6}$, 
which is impossible by Lemma \ref{N6Choosable}.

\claim{3} {$H$ contains no induced $C_4$.}
Suppose otherwise that $H$ contains an induced $C_4$, say $x_1x_2x_3x_4x_1$. 
Put $R \DefinedAs V(H) - \set{x_1, x_2, x_3, x_4}$.  Let $y \in R$. As
$\alpha(H) \leq 2$, $y$ has a neighbor in $\set{x_1, x_3}$ and a neighbor in
$\set{x_2, x_4}$.  If $y$ is adjacent to all of $x_1, \ldots, x_4$, then
$\join{K_1}{H}$ contains $\join{K_2}{C_4}$, which is $d_1$-choosable by Lemma~\ref{K2Classification}, and this is impossible.
If $y$ is adjacent to three of $x_1, \ldots, x_4$, then $\join{K_1}{H}$ contains
$\join{E_2}{\text{paw}}$, which is $d_1$-choosable by Lemma~\ref{E2JoinB}, and this is again impossible.

Thus every $y \in R$ is adjacent to all and only the vertices on one side of the
$C_4$.  We show that any two vertices in $R$ must be adjacent to the same or
opposite side and this gives the desired covering by two cliques.  If this
does not happen, then by symmetry we may suppose we have $y_1, y_2 \in R$ such
that $y_1 \adj x_1, x_2$ and $y_2 \adj x_2, x_3$.  We must have $y_1 \adj y_2$
for otherwise $\set{y_1, y_2, x_4}$ is an independent set.  But now $x_1y_1
y_2x_3x_4x_1$ is an induced $C_5$ in which $x_2$ has $4$ neighbors, impossible
by Claim~2.

\claim{4} {$H$ does not exist.}
By Claim~1 and Claim~3, $H$ contains an induced $C_5$.  By Claim~2, each vertex $y$ in $H$ that is not on the $C_5$ has at most 3 neighbors on the $C_5$.  Since $\alpha(H)\le 2$, each vertex $y$ in $H$ not on the $C_5$ must be adjacent to at least three consecutive vertices of the $C_5$.
This implies that $H$ is a thickening of $C_5$.
This final contradiction
completes the proof.
\end{proof}

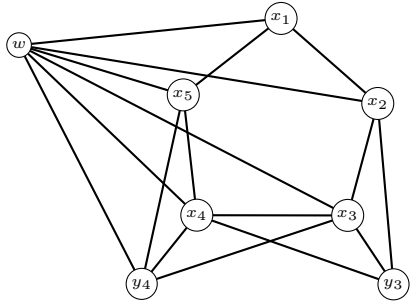
\begin{figure}[htb]
\centering
\begin{tikzpicture}[scale = 10]
\tikzstyle{VertexStyle}=[shape = circle,	
								 minimum size = 1pt,
								 inner sep = 1.2pt,
                         draw]
\Vertex[x = 0.535568237304688, y = 0.836275801062584, L = \tiny {$x_2$}]{v0}
\Vertex[x = 0.294997036457062, y = 0.687704384326935, L = \tiny {$x_4$}]{v1}
\Vertex[x = 0.495568662881851, y = 0.687386810779572, L = \tiny {$x_3$}]{v2}
\Vertex[x = 0.276996433734894, y = 0.847704276442528, L = \tiny {$x_5$}]{v3}
\Vertex[x = 0.556489169597626, y = 0.595672339200974, L = \tiny {$y_3$}]{v4}
\Vertex[x = 0.0587111786007881, y = 0.913989931344986, L = \tiny {$w$}]{v5}
\Vertex[x = 0.221746027469635, y = 0.596272975206375, L = \tiny {$y_4$}]{v6}
\Vertex[x = 0.407015770673752, y = 0.949415870010853, L = \tiny {$x_1$}]{v7}
\Edge[](v2)(v0)
\Edge[](v2)(v1)
\Edge[](v3)(v1)
\Edge[](v4)(v1)
\Edge[](v0)(v5)
\Edge[](v1)(v5)
\Edge[](v2)(v5)
\Edge[](v3)(v5)
\Edge[](v4)(v2)
\Edge[](v4)(v0)
\Edge[](v6)(v1)
\Edge[](v6)(v5)
\Edge[](v6)(v3)
\Edge[](v6)(v2)
\Edge[](v7)(v3)
\Edge[](v7)(v0)
\Edge[](v7)(v5)
\end{tikzpicture}
\caption{The graph $D_8$.}
\label{fig:D8}
\end{figure}

\begin{lem}\label{D8Choosable}
The graph $D_8$ is $d_1$-choosable.
\end{lem}
\begin{proof}
Suppose not and let $L$ be a minimal bad
$d_1$-assignment on $G \DefinedAs D_8$. 

\claim{1} {$|Pot(L)|\le 6$.} By the Small Pot Lemma, we know that
$\card{Pot(L)}\le 7$.  Suppose $\card{Pot(L)} = 7$.  Put $\set{a,b}:=Pot(L) - L(w)$.  

We must have $L(y_3) = \set{a,b}$. Otherwise we could color $y_3$ from 
$L(y_3) - \set{a,b}$ and note that $G-y_3-w$
is $d_0$-choosable and hence has a coloring from its lists.  Then we can easily
modify this coloring to use both $a$ and $b$ at least once.  But now we can
color $w$.

If there exist distinct vertices $u,v\in V(G)-y_3$ such that $a \in L(u)$, $b
\in L(v)$ and $\{u,v\}\not\subseteq \{x_2,x_3,x_4\}$, then we can color $G$ as
follows.  Color $y_3$ arbitrarily to leave $a$ available on $u$ and
$b$ available on $v$.  Again, $G-y_3-w$ has a coloring.  We can modify it to
use $a$ and $b$, then color $w$.  Thus, $a$ and $b$ each appear only on
some subset of $\{y_3,x_2,x_3,x_4\}$.  

If $a\in L(x_2)\cap L(x_4)$, then we use $a$ on $x_2$ and $x_4$ and color
greedily $y_3$, $x_3$, $y_4$, $x_1$, $x_5$, $w$ (actually any order will work
if $y_3$ is first and $w$ is last).   If $a$ appears only on $y_3$ and exactly
one neighbor $x_i$, then we violate Lemma \ref{ComponentsOfColor}
since $\card{Pot_{y_3, x_i}(L)} < 7$. So now $a$ appears precisely on either
$y_3,x_2,x_3$ or $y_3,x_3,x_4$. Similarly $b$ appears precisely on either
$y_3,x_2,x_3$ or $y_3,x_3,x_4$.

If $\{a,b\}\cap L(x_2)=\emptyset$, then we use $a$ on $y_3$ and $b$ on $x_3$,
then greedily color $y_4$, $x_4$, $x_5$, $x_1$, $w$, $x_2$.  By symmetry, we may
assume that $a \in L(x_2)$. But then since $\set{a,b} \subseteq L(x_3)$ we
have $\card{Pot_{y_3, x_2, x_3}(L)} < 7$, violating Lemma
\ref{ComponentsOfColor}.  Hence $\card{Pot(L)} \leq 6$.

\claim{2} {$|Pot(L)|\le 5$.}  Suppose $\card{Pot(L)}=6$.  Choose $a
\in Pot(L) - L(w)$ and $b \in L(w) \cap L(y_3)$.  Put $H \DefinedAs G - y_3 -
w$.

First we show that $b\in L(x_2)\cap L(x_3)\cap L(x_4)$.  If not, we use $b$ on
$y_3$ and $w$, then greedily color $x_1$, $x_5$, $y_4$.  Now we can finish by
coloring last the $x_i$ such that $b\notin L(x_i)$.

We must have $a \in L(y_3)$ or else we color $x_2, x_4$ with $b$ and something
else in $H$ with $a$ (since $G_a$ contains an edge by Lemma \ref{ComponentsOfColor}) and
finish.  Now $a \not \in L(x_1), L(x_5), L(y_4)$, for
otherwise we color $x_2, x_4$ with $b$, $y_3$ with $a$ and then color $x_1, x_5,
y_4, x_3$ in order using $a$ when we can, then color $w$.  Now $a$ is
on $y_3$ and at least two of $x_2, x_3, x_4$ or else we violate Lemma
\ref{ComponentsOfColor}.  Now $a \not \in L(x_2) \cap L(x_4)$ since otherwise we
color $x_2, x_4$ with $a$, then $y_3$ with $b$, then greedily color $x_1, x_5,
y_4, x_3, w$.  Also $a \not \in L(x_2) \cap L(x_3)$ since then $\set{a,b}
\subseteq L(y_3) \cap L(x_2) \cap L(x_3)$ and hence $\card{Pot_{y_3, x_2,
x_3}(L)} < 6$, violating Lemma \ref{ComponentsOfColor}.  Therefore $V(G_a) =
\set{y_3, x_3, x_4}$.

Now $\card{Pot_{y_3, x_3, x_4}(L)} \leq 6$ and hence
$L(x_3) \cap L(x_4) = \set{a,b}$ for otherwise we violate Lemma
\ref{ComponentsOfColor}.  Suppose $L(x_3) = \set{a,b,c,d}$ and $L(x_4) =
\set{a,b,e,f}$. Then by symmetry $L(x_1)$ contains either $c$ or $e$.  
If $c \in L(x_1)$, color $x_1, x_3$ with $c$, $x_4$ with $a$, and $y_3$ with $b$.
Now we can greedily finish.  If $e \in L(x_1)$, color $x_1, x_4$ with $e$,
$x_3$ with $a$, and $y_3$ with $b$; again we can greedily finish.  Hence
$\card{Pot(L)} \leq 5$.

\claim{3} {$L$ does not exist.}  Since $\card{Pot(L)} \leq 5$ we
see that $x_3, x_5$ have two colors in common and $x_2, x_4$ have two colors in
common as well. In fact, these sets of common colors must be the same and equal to
$L(y_3) \DefinedAs \set{a,b}$ or we can finish the coloring (by first coloring $y_3$, 
then invoking Lemma~\ref{IntersectionsInB}).   Similarly, we may
assume that $a \in L(y_4)$ (if $\{a,b\}\cap L(y_4)=\emptyset$, then we have
$L(x_2)\cap L(y_4)\cap (Pot(L)\setminus\{a,b\})\ne \emptyset$ and we have color $a$ on
$x_3, x_5$, so we can color $y_3$ with $b$ and then finish by
Lemma~\ref{IntersectionsInB}).
Similarly, $L(x_1)$ contains $a$ or $b$.  But it cannot
contain $a$ for then we could color $y_3, y_4, x_1$ with $a$, and $x_2, x_4$ with
$b$, and then finish greedily. Say $L(x_4) = \set{a,b,c,d}$. Then as no nonadjacent
pair has a color in common that is in $Pot(L) - \set{a,b}$ we have $L(x_2) =
\set{a,b,e}$, then by symmetry of $c$ and $d$ we have $L(x_5) = \set{a,b,c}$.
Then $L(x_3) = \set{a,b,d,e}$ and hence $L(x_1) = \set{a,b}$, which contradicts
that $a\notin L(x_1)$.
We conclude that $L$ cannot exist.
\end{proof}

\begin{lem}\label{TwoTwoOneTwoOne}
Let $H$ be a thickening of $C_5$ such that $\card{H} \geq 6$. Then
$\join{K_1}{H}$ is $f$-choosable, where $f(v) \geq d(v)$ for the $v$ in the $K_1$
and $f(x) \geq d(x) - 1$ for $x \in V(H)$.
\end{lem}
\begin{proof}
Suppose not and let $L$ be a minimal bad $f$-assignment on $\join{K_1}{H}$. By
the Small Pot Lemma, $\card{Pot(L)} \leq \card{H}$. Note that $H$ is
$d_0$-choosable since it contains an induced diamond. 
Let $x_1, \ldots, x_5$ be the vertices of an induced $C_5$ in $H$.  Then $\sum_i
\card{L(x_i)} = \sum_i d_H(x_i) = 3\card{H} - 5 > 2\card{H} \geq \omega(H[x_1,
\ldots, x_5])\card{Pot(L)}$ and hence some nonadjacent pair in $\set{x_1, \ldots, x_5}$ have a color in
common.  Now applying Lemma \ref{LowSinglePair} gives a contradiction.
\end{proof}

We are now in a position to finish the proof of the Borodin-Kostochka Conjecture for claw-free
graphs.

\begin{thm}\label{BKClawFree}
Every claw-free graph satisfying $\chi \geq \Delta \geq 9$ contains a
$K_\Delta$.
\end{thm}
\begin{proof}
Suppose not and choose a counterexample $G$ minimizing $\card{G}$.  Then $G$ is
vertex critical and not quasi-line by Theorem \ref{QuasiLineColoring}.  Hence $G$
contains a vertex $v$ that is not bisimplicial.  By Lemma
\ref{BisimplicialOrThickC5}, $G_v \DefinedAs G[N(v)]$ is a thickening of a
$C_5$.  Also, by Lemma \ref{TwoTwoOneTwoOne}, $v$ is high. Pick a $C_5$ in
$G_v$ and label its vertices $x_1, \ldots, x_5$ in clockwise order.  For $i \in \irange{5}$, let $T_i$ be the thickening clique
containing $x_i$.  Also, let $S$ be those vertices in $V(G) - N(v) - \set{v}$
that have a neighbor in $\set{x_1, \ldots, x_5}$.  First we establish a few
properties of vertices in $S$. 

\claim{1} {For $z \in S$ we have $N(z) \cap \set{x_1, \ldots, x_5} \in \set{\set{x_i,
x_{i+1}}, \set{x_i, x_{i+1}, x_{i+2}}}$ for some $i \in \irange{5}$.} 
Let $z \in S$ and put $N \DefinedAs N(z) \cap \set{x_1, \ldots, x_5}$.  If
$\card{N} \geq 4$, then some subset of $\set{v, z} \cup N$ induces the
graph $\join{E_2}{P_4}$, which is $d_1$-choosable by Lemma~\ref{E2JoinB}.  
Hence $\card{N} \leq 3$.  Since $G$ is
claw-free, the vertices in $N$ must be contiguous.

\claim{2} {If $z \in S$ is adjacent to $x_i, x_{i+1}, x_{i+2}$, then $\card{T_i} =
\card{T_{i+1}} = \card{T_{i+2}} = 1$.}
Suppose not. First, we deal with the case when $\card{T_{i+1}} \geq 2$.  Pick
$y \in T_{i+1} - x_{i+1}$.  If $y \nonadj z$, then $\set{x_i, y, z, x_{i-1}}$
induces a claw, impossible.  Thus $y \adj z$ and $\set{v, z, x_i, x_{i+1},
x_{i+2}, y}$ induces the graph $\join{E_2}{\text{diamond}}$, which is $d_1$-choosable by
Lemma~\ref{E2JoinB}.

Hence, by symmetry, we may assume that $\card{T_i} \geq 2$.  If $y \nonadj z$,
then $\set{v, x_1, \ldots, x_5, y, z}$ induces a $D_8$ contradicting Lemma
\ref{D8Choosable}.  Hence $y \adj z$ and $\set{v, z, x_i, x_{i+1},
x_{i+2}, y}$ induces the graph $\join{E_2}{\text{paw}}$, which is $d_1$-choosable by
Lemma~\ref{E2JoinB}.

\claim{3} {For $i \in \irange{5}$, let $B_i$ be the $z \in S$ with
$N(z) \cap \set{x_1, \ldots, x_5} = \set{x_i, x_{i+1}}$.  Then $B_i \cup B_{i+1}$ and $B_i \cup T_i \cup T_{i+1}$ both induce cliques for any
$i \in \irange{5}$.}
Otherwise there would be a claw.

\claim{4} {$\card{T_i} \leq 2$ for all $i \in \irange{5}$.}
Suppose otherwise that we have $i$ such that $\card{T_i} \geq 3$.  Put $A_i
\DefinedAs N(x_i) \cap S$. By Claim~2, $A_i \subseteq
B_{i-1} \cup B_i$ and $A_i$ is joined to $T_i$.  Thus $T_i$ is joined to $F_i
\DefinedAs \set{v} \cup A_i \cup T_{i-1} \cup T_{i+1}$.  If $A_i\ne\emptyset$,
then $F_i$ induces a graph that is connected and not almost complete, which
is impossible by Lemma \ref{K3Classification}.   If $A_i = \emptyset$, then
$x_i$ must have at least $\Delta - 2$ neighbors in $T_{i-1} \cup T_i \cup
T_{i+1}$.  But that leaves at most one vertex for $T_{i-2} \cup T_{i+2}$,
which is impossible.

\claim{5} {$G$ does not exist.}
Since $d(v) = \Delta \geq 9$, by symmetry we may assume that $\card{T_i} = 2$
for all $i \in \irange{4}$.  As in the proof of Claim~4, we get that $T_2$ is joined to $F_2$. Since $|T_i|\le 2$ for all $i$, we must have $A_i\ne \emptyset$ (for all $i$,
but in particular for $A_2$). Since $A_i\subseteq B_{i-1}\cup B_i$, by symmetry,
we may assume that $A_2 \cap B_2 \neq \emptyset$. Pick $z \in A_2 \cap B_2$ and $y_i \in T_i - x_i$ for $i \in \irange{3}$. 
Then $F_2$ has the graph in Figure \ref{fig:FinalContradiction} as an induced subgraph, but this is impossible by Lemma \ref{K2Classification}.
\end{proof}
  
\begin{figure}[htb]
\centering
\begin{tikzpicture}[scale = 10]
\tikzstyle{VertexStyle}=[shape = circle,	
								 minimum size = 1pt,
								 inner sep = 1.2pt,
                         draw]
\Vertex[x = 0.89000016450882, y = 0.730000019073486, L = \tiny {$x_1$}]{v0}
\Vertex[x = 0.410000026226044, y = 0.538000077009201, L = \tiny {$y_3$}]{v1}
\Vertex[x = 0.648000001907349, y = 0.631999969482422, L = \tiny {v}]{v2}
\Vertex[x = 0.409999966621399, y = 0.731999933719635, L = \tiny {$x_3$}]{v3}
\Vertex[x = 0.892000138759613, y = 0.537999987602234, L = \tiny {$y_1$}]{v4}
\Vertex[x = 0.198000028729439, y = 0.616000115871429, L = \tiny {$z$}]{v5}
\Edge[](v0)(v2)
\Edge[](v1)(v2)
\Edge[](v3)(v1)
\Edge[](v3)(v2)
\Edge[](v4)(v0)
\Edge[](v4)(v2)
\Edge[](v5)(v3)
\Edge[](v5)(v1)
\end{tikzpicture}
\caption{$K_2$ joined to this graph is $d_1$-choosable}
\label{fig:FinalContradiction}
\end{figure}
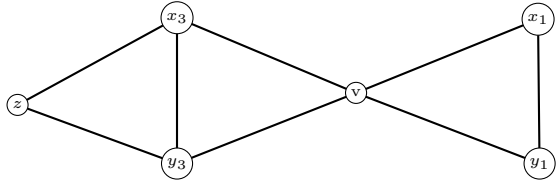

We note that this reduction to the quasi-line case also works for the
Borodin-Kostochka conjecture for list coloring; that is, we have the following
result.

\begin{thm}\label{ClawFreeLiftForLists}
If every quasi-line graph satisfying $\chi_l \geq \Delta \geq 9$ contains a
$K_\Delta$, then the same statement holds for every claw-free graph.
\end{thm}

\section{Line graphs}

In \cite{rabern2011strengthening}, the second author proved the
Borodin-Kostochka conjecture for line graphs of multigraphs.  Our aim
in this section is to lay out what we can prove about the list version of the
Borodin-Kostochka conjecture for line graphs (of multigraphs).  Our main result in this direction is Theorem 5.6, which says that if $H$ has minimum degree at least 7 and $G$ is the line graph of $H$, then the list version of
the Borodin-Kostochka Conjecture holds for $G$. 

\begin{thm}[Rabern \cite{rabern2011strengthening}]\label{BKLineGraph}
Every line graph of a multigraph satisfying $\chi \geq \Delta \geq 9$ contains a
$K_\Delta$.
\end{thm}

Some of the techniques used in the proof of this theorem carry over to the
Borodin-Kostochka conjecture for list coloring; unfortunately, a key part of the
proof used the fan equation and we do not have that for list coloring.  

\begin{lem}\label{BipartiteComplementJoin}
Fix $t \geq 2$ and $j \in \set{0, 1}$. Let $B$ be the complement of a bipartite
graph with $\omega(B) < |B| - j$. Let $L$ be a list assignment on $G \DefinedAs \join{K_t}{B}$ with $|L(v)| \geq d(v)
- j$ for each $v \in V(K_t)$ and $|L(v)| \geq d(v) - 1$ for each $v \in V(B)$. 
If $G$ is not $L$-colorable, then: 

\begin{itemize}
\item $t = 3$ and $B$ is the disjoint union of two complete subgraphs; or,
\item $t = 2$ and $B$ is the disjoint union of two complete subgraphs; or,
\item $t = 2$ and $B$ is formed by adding an edge between two disjoint complete
subgraphs; or,
\item $t = 2$, $B$ has a dominating vertex $v$ and $B - v$ is the disjoint
union of two complete subgraphs.
\end{itemize}
\end{lem}
\begin{proof}
If $t \geq 4$, then by Lemma \ref{ConnectedAtLeast4Poss}, $B$ is almost complete
and hence $j = 0$.  But then Lemma \ref{mixed} gives a contradiction.  Hence $t
\leq 3$.

Suppose $t = 3$. By Lemma \ref{K3Classification}, $B$ is either almost
complete or $\djunion{K_r}{K_{\card{B} - r}}$.  Suppose $B$ is almost complete.  Then $j
= 0$. Let $z \in V(B)$ be the vertex outside of the $|B| - 1$ clique and $x \in
V(B)$ some nonneighbor of $x$.  Then $\card{L(x)} + \card{L(z)} \geq d(x) +
d(z) - 2 = d_B(z) + |B| + 2$.  By the Small Pot Lemma (see
Section~\ref{ListLemmas}),
$\card{Pot(L)} \leq |B| + 2$.  Hence if $d_B(z) > 0$, we could color $x$ and
$z$ the same and then greedily complete the coloring to the rest of $G$,
impossible.  So, $B$ is $\djunion{K_1}{K_{|B| - 1}}$.

Now suppose $t = 2$.  If $B$ has no dominating vertex, then by Lemma
\ref{K2Classification}, $B$ is the disjoint union of two complete subgraphs or $B$ is formed by adding an edge between two disjoint complete
subgraphs.  Otherwise $B$ has a dominating vertex $v$ and hence $B =
\join{K_3}{B-v}$.  Similarly to the $t=3$ case, we conclude that $B - v$ is the
disjoint union of two complete subgraphs.
\end{proof}

\begin{lem}\label{muBoundLemma}
Let $H$ be a multigraph and let $G$ be the line graph of $H$ such that 
$\omega(G)<\chi_l(G)=\Delta(G)$.  Suppose we have a bad
$(\Delta(G)-1)$-assignment $L$ on $G$, and that $G$ is vertex critical with
respect to $L$.  Then $\mu(H) \leq 3$ and no
multiplicity $3$ edge is in a triangle.  Let $xy \in E(G)$ have $\mu(xy) = 2$. 
Then $xy$ is contained in at most one triangle.  Moreover, this triangle is
either $4$-sided or $5$-sided.  If the triangle is $5$-sided, then one of $x$ or
$y$ has all its neighbors in the triangle and in particular has degree at most
$4$ in $H$.

\end{lem}
\begin{proof}
Put $\Delta \DefinedAs \Delta(G)$. Let $xy \in E(H)$ be an
edge in $H$.  Let $A$ be the set of all edges incident with both $x$ and $y$.  
Let $B$ be the set of edges incident with either $x$ or $y$ but not both.  
Then, in $G$, $A$ is a clique joined to $B$ and $B$ is the complement of a
bipartite graph.  Put $F \DefinedAs G[A \cup B]$. Since $xy$ is $L$-critical,
we can color $G - F$ from $L$.   Doing so leaves a list assignment $J$ on $F$
with $|J(v)| = \Delta - 1 - (d_G(v) - d_F(v)) = d_F(v) - 1 + \Delta - d_G(v)$
for each $v \in V(F)$.  Put $j \DefinedAs d_G(xy) + 1 - \Delta$.  Since
$d_G(xy) + 1 = |A| + |B|$ and $\Delta > \omega(G) \geq \omega(A) + \omega(B) =
|A| + \omega(B)$, we have $\omega(B) < |B| - j$.

Therefore we may apply Lemma \ref{BipartiteComplementJoin}. We conclude $\mu(xy)
\leq 3$. Also, if $B$ is a disjoint union of two cliques in $G$, then $xy$ is in
no triangle.  Now suppose $\mu(xy) = 2$.  If $B$ has no dominating vertex in
$G$, then $xy$ is in exactly one triangle and it is $4$-sided.  Otherwise, by
symmetry we may assume that $B$ has a dominating vertex $xz$.  Then $y$ must
have all its edges to $x$ and $z$ and $y$ must have at least one edge to $z$
for otherwise we would have a $\Delta$ clique in $G$. Since $B - xz$ is the
disjoint union of two cliques, we must have $\mu(xz) = 1$.  Also $\mu(yz) \leq
2$ and hence $d_H(y) \leq 4$.
\end{proof}

\begin{lem}
Let $H$ be a multigraph and let $G$ be the line graph of $H$ such that 
$\omega(G)<\chi_l(G)=\Delta(G)$.  Suppose we have a bad
$(\Delta(G)-1)$-assignment $J$ on $G$, and that $G$ is vertex critical with
respect to $J$.  Then $H$ cannot have triple edges $uv$ and $vw$, such that 
$d(u)\ge 6$, $d(w)\ge 6$, and $d(v)\ge 7$ (or $d(v)\ge 6$ and every edge
incident to $v$ in $H$ is low in $G$).
\label{TwoTripleEdges}
\end{lem}
\begin{proof}
Assume the contrary and let $H$ be a counterexample. Recall from 
Lemma~2.2 above that the maximum edge multiplicity of $H$ is at most 3.

Let $a_1$, $a_2$, $a_3$ be three edges incident to $u$ but
not $v$; let $b_1$, $b_2$, $b_3$ be the edges incident to $u$ and $v$; let $c$
be incident to $v$ but not $u$ or $w$; let $d_1$, $d_2$, $d_3$ be incident to
$v$ and $w$; let $e_1$, $e_2$, $e_3$ be incident to $w$ (but not $u$ or $v$). 
We use these names for both the edges of $H$ and the vertices of $G$,
interchangeably.  

By criticality of $G$, we can color
$V(G)\setminus\{a_1,a_2,a_3,b_1,b_2,b_3,c,d_1,d_2,d_3,e_1,e_2,e_3\}$ from $J$. 
Let $L$ denote the list of remaining colors on the uncolored vertices.
Note that $|L(a_i)|\ge 4$, $|L(e_i)|\ge 4$, $|L(c)|\ge 5$, $|L(b_i)|\ge 8$, and
$|L(d_i)|\ge 8$.  We may assume that equality holds in each case.

\claim{1} {If there exist $\alpha\in L(a_1)\cap L(c)$, then we can
color $G$ from its lists.}  Suppose such an $\alpha$ exists.  
We use $\alpha$ on $a_1$ and $c$.  This saves a color on each
of $b_1$, $b_2$, $b_3$.  Now $|L(d_1)\setminus\{\alpha\}|+|L(a_2)\setminus
\{\alpha\}|\ge 7+3 > 8=|L(b_1)|$, so we can color $d_1$ and $a_2$ to save an
additional color on $b_1$.  Now we greedily color $e_1$, $e_2$, $e_3$, $d_2$,
$d_3$, $c$, $a_2$, $a_3$, $b_3$, $b_2$, $b_1$.

\claim{2} {If there exists $\alpha\in L(a_1)\cap L(d_1)$, then we
can color $G$ from its lists.}  Suppose such an $\alpha$ exists.  If $\alpha\in
L(c)$, then we proceed as above.  Otherwise we use $\alpha$ on  $a_1$ and $d_1$.
This saves a color on $b_1$, $b_2$, and $b_3$.  We may assume that $\alpha\in
L(b_1)$, since otherwise we can color greedily toward $b_1$.  Now we get
$|L(a_2)\setminus \{\alpha\}|+|L(c)|\ge 3+5 > 7=|L(b_1)\setminus\{\alpha\}|$.
Thus, we can color $a_2$ and $c$ to save a color on $b_1$.  Afterwards we color
greedily toward $b_1$.

\claim{3} {We may assume that
$L(b_1)=L(b_2)=L(b_3)=L(d_1)=L(d_2)=L(d_3)$; otherwise we can color $G$ from
its lists.} Suppose to the contrary that (by symmetry) there exists
$\alpha\in L(d_1)\setminus L(b_1)$.  If we also have $\alpha\notin L(b_2)$,
then we may use $\alpha$ on $d_1$ (color $a_1$ arbitrarily) and proceed as in
Claim 2.  So now we have $\alpha\in L(b_2)$.  By Claim 2 and symmetry, we have
$\alpha\notin \cup L(e_i)$.  Thus we use $\alpha$ on $d_1$ (without reducing
the $L(e_i)$).  Since we have $|L(c)\setminus\{\alpha\}| + |L(a_1)| >
|L(b_2)\setminus\{\alpha\}|$, we can color $c$ and $a_1$ to save a color on
$b_2$.  Now we color $d_2$ and $a_2$ to save a second color on $b_1$.  Finally,
we color greedily toward $b_1$.

\claim{4} {We can color $G$.}
By Claim 2, we know that $L(a_1)\cap L(d_1)=\emptyset$.  By Claim 3, we
know that $L(b_1)=L(d_1)$; thus, $L(a_1)\cap L(b_1)=\emptyset$.  By symmetry, we
get $L(a_i)\cap L(b_j)=\emptyset$ for all $i,j\in\{1,2,3\}$.  Now we can color
the $a_i$ arbitrarily, which saves 3 colors on each of the $b_i$.  Finally, we
color greedily towards $b_1$.  This proves the lemma.
\end{proof}

As an application of the lemma above, we show that if $H$ has minimum
degree at least 7 and $G$ is the line graph of $H$, then the list version of
the Borodin-Kostochka Conjecture holds for $G$.  We need the following theorem,
due to Borodin, Kostochka, and Woodall.

\begin{thm}[Borodin, Kostochka, Woodall~\cite{BorodinKostochkaWoodall}]
\label{BKWTheorem}
Let $G$ be a bipartite multigraph with partite sets $X$, $Y$, so that $V=X\cup
Y$.  $G$ is edge-$f$-choosable, where $f(e):=\max\{d(x),d(y)\}$ for each edge
$e=xy$.
\end{thm}

\begin{thm}
Let $H$ be a multigraph with $\delta(H)\ge 7$ and let $G$ be the line graph of
$H$.  Then $\chi_l(G)\le \max\{\omega(G),\Delta(G)-1\}$.
\label{ListLineGraphs}
\end{thm}
\begin{proof}
Suppose the contrary, and let $G_0$ (and $H_0$) be a counterexample with list
assignment $L$.  Let $G$ (and $H$) be a vertex critical subgraph with respect
to $L$.
It suffices to color $G$ from $L$.  Note that $\Delta(G)=\Delta(G_0)$, since
otherwise we can color $G$ from $L$ by the list version of Brooks' Theorem.  
Since $G$ is $L$-critical, we have $\delta(G)\ge \Delta(G)-1$.  Thus, we have
$d_H(u)\ge d_{H_0}(u)-1$ for all $u\in V(H)$ so $\delta(H)\ge 6$.  In
particular, if $uv$ is high in $G$, then $d_{H}(u)=d_{H_0}(u)$ and
$d_{H}(v)=d_{H_0}(v)$.  Note that if $\mu_H(xy)=3$, then in $H_0$ each of $x$
and $y$ is incident to only one triple edge (by Lemma~\ref{TwoTripleEdges},
since $G$ is critical with respect to $L$).

\claimnonum{If $xy$ is an edge of $H$ with multiplicity 3 and $d(x)=7$,
then vertex $xy$ is low in $G$.}
Suppose the contrary.  Since $G$ is $\chi_l$-critical, for each edge $xy\in H$,
we have $d_H(x)+d_H(y)=\Delta(G)+\mu(xy)+1$ if $xy$ is high and
$d_H(x)+d_H(y)=\Delta(G)+\mu(xy)$  if $xy$ is low.  Suppose that
$\mu(xy)=3$, $d(x)=7$, and $xy$ is high.  Now we get $d_{H_0}(y)=\Delta(G)-3$. 
By the last sentence of the previous paragraph, we know that every
edge incident to $y$ other than $xy$ has multiplicity at most 2.  Let $z$ be a
neighbor of $y$ other than $x$.  By the
degree condition above, we get $d_{H_0}(y)+d_{H_0}(z) \le \Delta(G)+\mu_{H_0}(xy) +1
\le \Delta(G) + 3$.  This implies that $d_{H_0}(z)\le 6$, which is a
contradiction.  This proves the claim.  More simply, for any vertex $x$ with
$d_{H}(x)=6$, we see that every edge $xy$ incident to $x$ must be low in $G$.

Now for each triple edge of $H$ that is high in $G$, we delete one of the
edges; call the resulting graph $G'$ (and $H'$).  Clearly, we have
$\delta(H')\ge 6$.  By the previous Lemma and the claim, $d_{H'}(x)\ge 7$ for
every vertex $x$ incident to a triple edge or an edge $xy$ that corresponds to
a high vertex in $G$.  For if $xy$ is a triple edge and $d(x)=7$, then edge
$xy$ is low in $G$.  Similarly, if $d(x)=6$, then every edge $xy$ is low in
$G$.  Otherwise, each vertex of $H$ that is incident to a triple edge has
degree at least 8 and is incident to exactly one triple edge.  

Let $B$ be a maximum bipartite subgraph of $H'$.  For each vertex $x\in H$, we
have $d_{B}(x)\ge d_{H'}(x)/2$ (otherwise $B$ has more edges if we move $x$ to
the other part); thus $\delta(B)\ge 3$ and $d_B(x)\ge 4$ for each vertex
incident to a triple edge or an edge $xy$ that is high in $G$.  
Let $G_B$ denote the line graph of $B$.  
Since $G$ is critical with respect to $L$, we can color $G-V(G_B)$ from $L$.
So to show that $G$ is $(\Delta-1)$-choosable, it suffices to show that we
can color $G_B$ when each vertex $v$ that is high in $G$ gets a list of size
$d_{G_B}(v)-1$ and each vertex $v$ that is low in $G$ gets a list of size
$d_{G_B}(v)$.  Consider a high vertex $xy$ in $G$; recall that $d_{H'}(x)\ge 7$
and $d_{H'}(y)\ge 7$, so $d_B(x)\ge 4$ and $d_B(y)\ge 4$.  The degree of $xy$
in $G_B$ is $d_{B}(x)+d_{B}(y)-\mu_B(xy)-1$.  Since $\mu_B(xy)\le 2$, we see
that $xy$ has a list of size at least $d_B(x)+d_B(y)-4\ge \max\{d_B(x),d_B(y)\}$.
 Each low vertex $xy$ has a list of size at least $d_B(x)+d_B(y)-\mu(xy)-1\ge
\max\{d_B(x),d_B(y)\}$.
So by Theorem~\ref{BKWTheorem}, we can color $G_B$ from its remaining lists.
\end{proof}

Theorem~\ref{ListLineGraphs} is best possible in the sense that if we replace ``$\delta(H)\ge 7$''
by ``$\delta(H)\ge 6$'', then the theorem is false.  One counterexample is when
$H$ is a 5-cycle in which each edge has multiplicity 3, shown in Figure~1 (d).

\bibliographystyle{amsplain}
\bibliography{GraphColoring}

\end{document}